\documentclass{amsart}

\usepackage[english]{babel}
\usepackage{microtype}
\usepackage{amsmath}
\usepackage{amscd}
\usepackage{amssymb}
\usepackage{amsthm}
\usepackage[table]{xcolor}
\usepackage[matrix, arrow,curve]{xy}
\usepackage{verbatim}

\usepackage{soul}

\usepackage{tikz}
\usepackage{tikz-cd}
\usetikzlibrary{arrows}

\makeatletter
\def\pgf@stroke@inner@line{%
 \let\pgf@temp@save=\pgf@strokecolor@global
 \pgfsys@beginscope%
 {%
  \pgfsys@roundcap% <-- I add this
  \pgfsys@setlinewidth{\pgfinnerlinewidth}%
  \pgfsetstrokecolor{\pgfinnerstrokecolor}%
  \pgfsyssoftpath@invokecurrentpath%
  \pgfsys@stroke%
 }%
 \pgfsys@endscope%
 \global\let\pgf@strokecolor@global=\pgf@temp@save
}
\makeatother

\usepackage[inline]{enumitem}

\usepackage{xr} % for cross-references between the two parts
\usepackage{hyperref}
\usepackage[capitalise,nosort,nameinlink]{cleveref}
\usepackage{thmtools} %for autoref to work properly
\usepackage{txfonts}

\definecolor{darkgreen}{rgb}{0.0, 0.2, 0.13}

\def\IN{\mathbb N}
\def\IZ{\mathbb Z}
\def\IR{\mathbb R}

\def\IQ{\mathbb Q}

\def\IB{\mathbb B}

\def\an{\mathrm{an}}
\def\exp{\mathrm{exp}}

\def\vs{\vspace}
\def\ma{\mathcal}

\def\alg{\mathrm{alg}}

\newcommand{\R}{\IR}

\newcommand{\Eh}{\mathcal{E}}
\renewcommand{\div}{\mathrm{div}}
\newcommand{\vol}{\mathrm{vol}}
\newcommand{\sub}{\mathrm{sub}}
\newcommand{\con}{\mathrm{con}}
\newcommand{\defin}{\mathrm{def}}
\newcommand{\sing}{\mathrm{sing}}
\newcommand{\dR}{\mathrm{dR}}

\newcommand{\Mh}{\mathcal{M}}
\newcommand{\Ah}{\mathcal{A}}

\newcommand{\Cdef}{\mathcal{C}_\mathrm{def}}
\newcommand{\Csub}{\mathcal{C}_\mathrm{sub}}
\newcommand{\Ccon}{\mathcal{C}_\mathrm{con}}

\newcommand{\Oconext}[2]{\Omega^{#1}_{\mathrm{con}(#2)}}

\newtheorem{lemma}{Lemma}[section]
\newtheorem{prop}[lemma]{Proposition}
\newtheorem{thm}[lemma]{Theorem}
\newtheorem{cor}[lemma]{Corollary}
\newtheorem{theorem}[lemma]{Theorem}

\theoremstyle{definition}
\newtheorem{defn}[lemma]{Definition}
\newtheorem{notation}[lemma]{Notation}

\newtheorem{ex}[lemma]{Example}

\newtheorem{rem}[lemma]{Remark}

\newtheorem{defnrem}[lemma]{Definition and Remark}
\newtheorem{finrem}[lemma]{Final Remark}

\begin{document}
\title[On the De Rham Theorem in the Globally Subanalytic Setting]{On the De Rham Theorem in the Globally Subanalytic Setting}
\author{Annette Huber}
\address{Mathematical Institute,
University of Freiburg,
79085 Freiburg,
Germany}
\email{ annette.huber@math.uni-freiburg.de}
\author{Tobias Kaiser}
\address{Faculty of Computer Science and Mathematics,
University of Passau,
94030 Passau,
Germany
}
\email{tobias.kaiser@uni-passau.de
}
\author{Abhishek Oswal}
\address{Mathematical Institute,
University of Freiburg,
79085 Freiburg,
Germany}
\email{abhishek.oswal@math.uni-freiburg.de}

\date{\today}

\begin{abstract}	
				For globally subanalytic manifolds we define de Rham complexes of globally subanalytic differential forms and of constructible differential forms. Whereas the de Rham theorem does not hold for the former in the non-compact case, it does hold for the latter in full generality. We deduce that the constructible de Rham cohomology groups are canonically isomorphic to the classical ones. We stress that our results apply already in the $C^1$-setting.
\end{abstract}
\maketitle	
%\tableofcontents	

	\section*{Introduction}
We are interested in homology and cohomology theories in the tame setting of o-minimal structures.
O-minimal singular homology and cohomology have been successfully developed, first by Delfs and Knebusch \cite{delfs_knebusch} 
in the semialgebraic case and then by Edmundo and Woerheide \cite{edmundo_woerheide_comparison} in the general case of o-minimal expansions of a real closed field.

What has been missing so far is an o-minimal version of de Rham cohomology, more precisely of the de Rham theorem comparing definable de Rham with o-minimal singular cohomology. Note that Bianconi and Figueiredo have previously introduced in their preliminary note \cite{bianconi2019minimal}, a version of definable de Rham cohomology groups.
The first author of the present paper has in \cite{huber_period_iso_tame} studied the period pairing for manifolds definable in an o-minimal structure, analyzing which simplices can be used for integration. But on the side of differential forms an o-minimal version leading to a definable de Rham theorem did not exist so far. 

We focus on the globally subanalytic setting, i.e. on the o-minimal structure $\IR_\an$ of the real field  with restricted analytic functions. Given a globally subanalytic $C^\omega$-manifold that is compact every real analytic function is globally subanalytic. Since in the $C^\omega$-setting the classical de Rham cohomology groups are generated by real analytic differential forms (see for example Beretta \cite{B}) we have immediately a definable version of the de Rham theorem. But in the non-compact case the situation is completely different. 
The straightforward definition using definable forms does not give rise to a good de Rham cohomology theory, mainly due to the failure of the Poincar\'e lemma. Integration arguments are no longer available in the o-minimal setting.

However in the globally subanalytic setting we are able to remedy the situation. Here one can consider the class of constructible functions as introduced by Raf Cluckers and Dan Miller \cite{cluckers_miller}. Constructible functions are finite sums of finite products of globally subanalytic functions and the logarithm of positive such functions. They are definable in the o-minimal structure $\IR_{\an,\exp}$ of the real field with restricted analytic functions and exponentiation (see van den Dries and Miller \cite{vdD_Miller_Geometric} for the structures $\IR_\an$ and $\IR_{\an,\exp}$). The main property of this class is that it is closed under parametric integration. 

Constructible functions of class $C^\infty$ are real analytic by \cite{kaiser_rocky}. The reasoning of \cite{B} involves deep results from  cohomology of coherent sheaves on Stein complex manifolds which are so far not available in the definable or constructible setting. 
But in the case of globally subanalytic $C^p$-manifolds where $0<p<\infty$ we are able to establish a full constructible de Rham theorem.
Given $0\leq q<p$ we define constructible differential forms of regularity $q$ and the corresponding de Rham cohomology groups. 
Note that this includes the case of $C^1$-manifolds: By the tame geometric properties of o-minimal structures we are able to establish a version of the de Rham complex leading to the de Rham theorem also for those.

The definition of  constructible differential forms is subtle. 
We construct them similarly to the definition of regulous functions (see Kucharz and Kurdyka \cite{KK_regulous}). We indicate the case $q=0$: 
A constructible differential form is $C^1$ on a dense open globally subanalytic set. We consider now the vector space of the continuous constructible differential forms such that the Cartan derivative on such an open set extends to a continuous 
constructible differential form on the whole manifold. This defines the \emph{constructible de Rham complex} and the \emph{constructible de Rham cohomology groups}. It is a consequence of the tame properties of o-minimal structures that they are indeed functorial for globally subanalytic $C^1$-maps.

Our main result, \autoref{T 6.4} and \autoref{thm:main-sheaf-version}, is the de Rham theorem for these constructible de Rham cohomology groups: they are canonically isomorphic to the singular cohomology groups (and if the given manifold is $C^\omega$ therefore to the classical de Rham cohomology groups). 

Following closely the arguments in the classical case, we give two versions of proof, one fairly explicit one depending on Stokes' Theorem and a sheaf theoretic one. In both cases the ingredients are the same:
the {\it Mayer-Vietoris property} and 
{\it homotopy invariance}.

The Mayer-Vietoris property follows in a straight-forward way from the existence of globally subanalytic partitions of unity in the $C^p$-case for $p<\infty$. It is this part of the argument that breaks down for the $C^\omega$-versions of constructible de Rham cohomology.

Homotopy invariance turns out to be more subtle. As for functoriality, tameness is used heavily.

The paper is organized as follows. In the first section we introduce some notations and preliminaries on definable manifolds and differential forms.  Note that we start with a fairly general setup to be used in future work. In Section 2 we construct the definable and constructible de Rham cohomology groups.
In the third section we show that we have functorial pull back and that Stokes's theorem holds in this setting. In Section 4 we establish homotopy invariance of constructible de Rham cohomology. In the fifth section we provide some geometric results allowing us to prove in Section 6 a constructible de Rham theorem in the case of globally $C^p$-manifolds with $p<\infty$ using the period pairing. In the last section we give a sheaf-theoretic proof. Here we also establish that the hypercohomology of the constructible de Rham complex computes the singular cohomology of a (not necessarily compact) globally subanalytic $C^\omega$-manifold.

\subsection*{Acknowledgements}
We would like to thank Nadine Große for discussion on the divergence lemma (Proposition~\ref{prop:analysis}.)

\section{Setting} 
\subsection{Notations and Preliminaries}\label{1.1}
\begin{itemize}
 \item By $\IN=\{1,2,3,\ldots\}$ we denote the set of natural numbers and by $\IN_0$ the set of natural numbers with $0$.
\item Given a set $A$ and some $a\in A$ we denote by $c_a$ the constant map on $A$ with value $a$.
\item  By $I$ we denote the open unit interval $(0,1)$ in $\IR$.
\item Let $r\in \IN_0\cup\{\infty,\omega\}$. We denote by $C^r$ the class of real valued functions which are $r$-times continuously differentiable  in the case $r\in \IN_0$ respectively infinitely differentiable in the case $r=\infty$ respectively analytic in the case $r=\omega$.
\item We denote by $\R_\alg$ the o-minimal structure of semialgebraic sets.
\item We denote by $\R_\an$ the o-minimal structure generated by the restricted analytic functions; the definable sets and functions are the globally subanalytic ones. 	
\end{itemize}

Throughout the paper we fix an o-minimal structure $\mathcal{M}$ expanding the ordered field of real numbers $(\R, <, +, \cdot, 0, 1)$, see \cite{van-den-Dries-tame-topology}. By definable, we shall always mean $\mathcal{M}$-definable (with parameters in $\R$).
 
We will be mostly interested in the case that $\mathcal{M}$ is the o-minimal structure $\R_{\an}$.	We assume familiarity with the definition of globally subanalytic sets and functions on the reals (compare with \cite{vdD_Miller_Geometric}).

Let $A$ be a globally subanalytic set in $\IR^n$. A \emph{constructible function} $f:A\to \IR$ is of the form
	$$f=\sum_{i=1}^k f_i\prod_{j=1}^{l_i} \log(g_{ij})$$
	where $f_i:A\to \IR, g_{ij}:A\to \IR_{>0}$ are globally subanalytic functions (see \cite{cluckers_miller}).
	Note that constructible functions are closed under composition with globally subanalytic functions from the right. The class is closed both under taking partial derivatives and parametric integrals. 
	The latter is the main result of Cluckers and Miller \cite{cluckers_miller}. The former follows by adapting the reasoning of \cite{kaiser_rocky} (where the corresponding result for the larger class of log-analytic functions has been shown) to the setting of constructible functions, see \cite{kaiser-ConstrDiff} for details.

\subsection{Definable manifolds}

We introduce the main objects of our paper.

\begin{defn} Let $p \in \IN_0\cup\{\infty, \omega\}$ and $X$ a Hausdorff topological space. Let $n\in \IN_0$.
 \begin{enumerate}
     \item A \emph{definable $C^p$-atlas $\mathcal{A}$ of constant dimension $n$} on $X$ is the data \[\mathcal{A} = \{(U_i, V_i\subset \R^n, \phi_i: U_i \xrightarrow{\sim} V_i) \mid 1 \leq i \leq r\},\] consisting of a finite open cover $X = \bigcup_{i=1}^r U_i$ of $X$,  definable open subsets $V_i \subset \R^{n}$ along with homeomorphisms $\phi_i : U_i \xrightarrow{\sim} V_i \subset \R^{n}$ such that for each $i,j$,  $\phi_i \circ \phi_j^{-1} : \phi_j(U_i \cap U_j) \xrightarrow{\sim} \phi_i(U_i \cap U_j)$ is a definable $C^p$-diffeomorphism. We refer to the maps $\phi_i$ as \emph{charts}.
     
     \item We declare two definable $C^p$-atlases $\mathcal{A} = \{(U_i, V_i, \phi_i) : 1\leq i \leq r\}$ and $\mathcal{A}' = \{(U_j', V_j', \phi_j') : 1 \leq j \leq s\}$ on $X$ of constant dimension $n$ to be equivalent if the union $\mathcal{A}\cup \mathcal{A}'$ is a definable $C^p$-atlas of constant dimension $n$ on $X$. This defines an equivalence relation on the set of all definable $C^p$-atlases on $X$ of dimension $n$. We denote the equivalence class of a definable $C^p$-atlas $\mathcal{A}$ on $X$  by $[\mathcal{A}].$ 
     
     \item A \emph{definable $C^p$-manifold of dimension $n$} is the data $(X, [\mathcal{A}])$ of a Hausdorff topological space $X$ equipped with the equivalence class of a definable $C^p$-atlas $[\mathcal{A}]$ of constant dimension $n$ on $X$. We shall often suppress the choice of the equivalence class of the atlas $[\mathcal{A}]$ for notational simplicity.
     
     \item Let $(X,[\mathcal{A}])$ be a definable $C^p$-manifold of dimension $n$. A subset $A \subset X$ is said to be \emph{a definable subset of $X$} if for every $\phi_i : U_i \xrightarrow{\sim} V_i \subset \R^n$ in $\mathcal{A}$, the set $\phi_i(A\cap U_i) \subset \R^n$ is a definable subset of $\R^n$. It is easy to check that the notion of definable subsets only depends on the equivalence class of the chosen atlas $\mathcal{A}$, and that the collection of definable subsets of $X$ is closed under finite intersections, finite unions, and complements. We shall denote the Boolean algebra of definable subsets of $X$ by $\mathrm{Def}(X)$.	
     
     \item A \emph{definable $C^p$-map} $f : (X,[\mathcal{A}]) \rightarrow (X',[\mathcal{A}'])$ of two definable $C^p$-manifolds $X$ and $X'$ is a continuous map $f : X \rightarrow Y$ such that for every element $(U_i, V_i, \phi_i : U_i \xrightarrow{\sim} V_i) \in \mathcal{A}$ and $(U_j',V_j',\phi_j': U_j' \xrightarrow{\sim} V_j') \in \mathcal{A}'$ the map $ \phi_j' \circ f \circ \phi_i^{-1} :  \phi_i\big(U_i \cap f^{-1}(U'_j)\big) \rightarrow V_j'$ is definable and $C^p$. This definition also only depends on the equivalence class of the chosen atlas $\mathcal{A}$.
     
	\item We denote by $C^p_\defin(X)$ the $\R$-algebra of definable $C^p$-functions on a definable $C^p$-manifold $X$.

	\item Let $X$ be a definable  $C^p$-manifold of dimension $n$ and let $d\in \{0,\ldots,n\}$. A subset $Y$ is a \emph{definable $C^p$-submanifold of $X$ of dimension $d$} if it is a definable subset and a $C^p$-submanifold of $X$ of dimension $d$. Note that $Y$ can then be endowed with a definable $C^p$-manifold structure.

\item A definable $C^p$-manifold $X$ is called \emph{affine} if it is definably $C^p$-diffeomorphic to a definable $C^p$-submanifold of $\R^N$ for some $N$.
 \end{enumerate} 
\end{defn}

\begin{rem}
    When $\mathcal{M} = \R_\mathrm{alg}$ 
 (or $\mathcal{M} = \R_\mathrm{an}$) a definable $C^p$-manifold shall also be called a \emph{semialgebraic  $C^p$-manifold}   (or \emph{a globally subanalytic $C^p$-mani\-fold}, respectively). A semialgebraic $C^\omega$-manifold is also referred to as a Nash manifold. We also note that $C^\infty$-semialgebraic and $C^\infty$-globally subanalytic  manifolds are automatically $C^\omega$ as follows from a remarkable result of van den Dries--Miller \cite{vdD_Miller_Tamm}. 
\end{rem}

In the globally subanalytic case, we are also interested in the constructible version.

\begin{defn}
Let $(X,[\Ah])$ be a globally subanalytic $C^p$-manifold.
A function $f : X \rightarrow \R$ is said to be \emph{constructible} if for every chart $(U_i,V_i\subset \R^n,\phi_i : U_i \xrightarrow{\sim} V_i) \in \mathcal{A}$, the composition $f \circ \phi_i^{-1} : \phi_i(U_i) \rightarrow \R$ is constructible. 

We let $\mathcal{C}_\mathrm{con}^p(X)$ be the $\R$-algebra of constructible $C^p$-functions on $X$.
\end{defn} 

Note that constructibility is invariant under change of globally subanalytic variables because a function on a globally subanalytic set is constructible if and only if its restriction to a finite cover by (not necessarily open) globally subanalytic subsets is constructible.

\begin{comment}
	Recall the following notions.
	
	\begin{defn}\label{definably-regular}
		A definable $C^p$-manifold $X$ is said to be \emph{definably regular} (respectively \emph{definably normal}) if for every point $x \in X$, (respectively every closed definable subset $C \subset X$) and every closed definable subset $C' \subset X$ such that $x \notin C'$ (respectively, $C\cap C' = \emptyset$) there exist disjoint open definable subsets $U$ and $U'$ such that $x \in U$ (respectively $C \subset U$) and $C' \subset U'$.
	\end{defn}	
\end{comment}

\begin{rem}\label{rem:definably-normal} 
    We note that every definable $C^0$-manifold is \emph{definably normal}  (see for instance \cite[Lemma 3.5]{edmundo2000minimal}, and \cite[Theorem 2.7]{bianconi2019minimal}). Recall this means that for two disjoint closed definable subsets $C, C' \subset X$ there exist disjoint open definable subsets $U \subset X$ and $U' \subset X$ such that $C \subset U$ and $C' \subset U'$.
    On the other hand, note that there are examples of Hausdorff definable spaces (in the sense of \cite[Ch. 10, \S 1]{van-den-Dries-tame-topology}) that are not definably normal (see \cite[Example on p.\,159]{van-den-Dries-tame-topology}).
    %Further, every definably regular $C^0$-manifold $X$ can be definably continuously embedded in some Euclidean space $\R^N$ (see for instance \cite[Chapter 10, Theorem 1.8]{van-den-Dries-tame-topology}). Therefore, every $C^0$-manifold is definably metrizable, and is hence also definably normal (see \cite[Theorem 2.7]{bianconi2019minimal}). 
\end{rem}

\begin{prop}[Definable Partition of Unity]\label{partition-of-unity}
    Suppose that $\mathcal{M}$ is a polynomially bounded o-minimal structure admitting $C^\infty$-cell decomposition. 
    Let  $p < \infty$ and let $X$ be a definable $C^p$-manifold. Then definable partitions of unity of order $p$ exist: given any open definable subset $U \subset X$ and any finite cover of $U$ by open definable subsets $U = \bigcup_{i=1}^r U_i$ one can find  definable $C^p$-functions $f_i:U\to \R$ for $ 1\leq i \leq r$ such that
    \begin{enumerate}
        \item  $\sum_{i=1}^r f_i = 1$, and
        \item the support of each $f_i$ is contained in some closed definable subset of $U$ contained in $U_i$.
    \end{enumerate}
\end{prop}
\begin{proof}    Let $U = \bigcup_{i=1} U_i$ be a finite definable open cover of a definable open subset $U\subset X$. Since $U$ is definably normal (see \autoref{rem:definably-normal}), we may replace $X$ by $U$ and assume that $X = U = \bigcup_{i=1}^r U_i$ is a definable open cover of $X$. Further replacing the cover $\{U_i : 1\leq i \leq r\}$ by a finite refinement, we shall assume without loss of generality that each $U_i$ is contained in a chart of some definable atlas. Since $X$ is definably normal, we may proceed as in \cite[Ch. 6, Lemma 3.6]{van-den-Dries-tame-topology}, to find open definable subsets $V_i \subset U_i$ such that the closure in $X$ of $V_i$ is contained in $U_i$ and such that $X = \bigcup_{i=1}^r V_i$. 

Our $X$ satisfies the assumptions of \cite[Theorem 6.1]{KCBP}.
It provides us with definable $C^p$-functions $\psi_i:X\to\R$ such that $\psi_i \geq 0$ and $\psi_i^{-1}(0) = X\setminus V_i$. Set 
$$f_i := \frac{\psi_i}{\sum_{j=1}^r \psi_j} \in \Cdef^p(X).$$ 
Then $\{f_i : 1\leq i \leq r\}$ is a partition of unity subordinate to the cover $X = \bigcup_{i=1}^r U_i.$ This completes the proof.
\end{proof}

Note that the assumptions of the proposition are satisfied in the cases
of $\R_\an$ and $\R_\alg$. 

\begin{rem}
% Kawakami stated in \cite{Kawakami} that every definable $C^r$-manifold where $r\geq 2$ is definably $C^r$-imbeddable into some euclidean space; i.e. it is affine not only as a definable space but as a definable manifold. 
In \cite[Section 4]{Fujita_Kawakami} Fujita and Kawakami %admit that there is a gap in \cite{Kawakami} and 
prove in the more general setting of definably complete locally o-minimal structures that for $r\geq 1$ a definably normal $C^r$-manifold is affine. In the course of their proof, they have formulated a result on the existence of definable partitions of unity. %Note that a definable space is definably normal if and only if it admits definable partition of unity.
For convenience, we have also given a direct proof above.
\end{rem}
%%%%%%%%%%%%%%%%%%%%%%%%%%%%%%%%%%%%%
\subsection{Homotopies}
We fix $p\in \IN_0\cup\{\infty,\omega\}$. Let $X,Y$ be definable $C^p$-manifolds.
\begin{defn}
Let $f,g:X\to Y$ be  definable $C^p$-maps. Then $f,g$ are called \emph{definably $C^p$-homotopic} if there is a definable  $C^p$-map $h:X\times (-\varepsilon,1+\varepsilon)\to Y$ for some $\varepsilon>0$ such that $f=h\vert_{X \times \{0\}}$ and $g=h\vert_{X \times \{1\}}$.
\end{defn}
\begin{rem}As all open (bounded or unbounded) intervals are even Nash-diffeomorphic, the particular choice of interval $(-\varepsilon,1+\varepsilon)$ is not
important in this definition. 
\end{rem}

\begin{defn}
	We say that $X$ and $Y$ are \emph{definably $C^p$-homotopy equivalent} if there are definable $C^p$-maps $\varphi:X\to Y$ and $\psi:Y\to X$ such that $\psi\circ\varphi$ is definably $C^p$-homotopic to $\mathrm{id}_X$ and 
	$\varphi\circ\psi$ is definably $C^p$-homotopic to $\mathrm{id}_Y$.
\end{defn}

\begin{defn}
	We call $X$ \emph{definably $C^p$-contractible} if it is definably $C^p$-homotopy equivalent to a point, in other words: if there is $a\in X$  such that $\mathrm{id}_X:X\to X$ and the constant function $c_a:X\to X$ with value $a$ are definably $C^p$-homotopic.
\end{defn}

%%%%%%%%%%%%%%%%%%%%%%%%%%%%%%%%%%%%%%%%%%%%%	
\subsection{Differential forms}
We assume familiarity with the calculus of differential forms on manifolds as in \cite{bott_tu}, \cite{madsen_tornehave} or \cite{warner-foundations}.

Let $X$ be a definable $C^p$-manifold where $p\in \IN\cup\{\infty,\omega\}$; i.e. $p\geq 1$. Then its cotangent bundle $T^*X \rightarrow X$ has a natural structure of a definable $C^{p-1}$-manifold such that the projection $T^*X \rightarrow X$ is a definable map of $C^{p-1}$-manifolds.
We denote by $\Lambda^kT^*X \rightarrow X$ its $k^\mathrm{th}$-exterior power. Again it has a natural structure of a definable $C^{p-1}$-manifold.

For the rest of the paper we use the convention $\infty-1:=\infty$ and $\omega-1:=\omega$. Let $0\leq q\leq p-1$.

\begin{defn}\label{def:definable-Cq-forms}
A \emph{definable $C^q$-form on $X$ of degree $k$} is a definable $C^q$-section of $\Lambda^kT^*X \rightarrow X$. We denote by $\Eh^k_{\defin(q)}(X)$ the space of all definable $C^q$-forms on $X$ of degree $k$.
\end{defn}

It follows from the definitions that a differential form $\omega$ of class $C^q$ is definable if all coefficient functions in all coordinate charts of a definable atlas are definable.

In the case of globally subanalytic manifolds we also write $\ma{E}^k_{\sub(q)}(X)$. There is also a constructible variant.

\begin{defn}\label{def:const-Cq-forms} Assume that $X$ is globally subanalytic. A
$C^q$-form $\omega$ on $X$ is called \emph{constructible} if all coefficient functions in all coordinate charts of a globally subanalytic atlas are constructible. We denote by $\Eh^k_{\con(q)}(X)$ the space of all constructible $C^q$-forms on $X$.
\end{defn}

\begin{rem}
Assume that $X$ is globally subanalytic. Let 
$\omega$  be a $C^\infty$-differential form on $X$.
	\begin{itemize}
		\item[(1)] If $\omega$ is globally subanalytic then $\omega$ is $C^\omega$ by \cite{vdD_Miller_Tamm}.
		\item[(2)] If $\omega$ is constructible then $\omega$ is $C^\omega$ by \cite{kaiser_rocky}.
	\end{itemize}
\end{rem}

Differential forms come with an algebra structure that will not play a role in our paper. On the other hand, the differential is essential. 

\begin{rem}\label{rem:exterior-der-of-definable-forms} Let $1\leq q\leq p-1$.
The differential induces a well-defined map
\[ d:\Eh^k_{\defin(q)}(X)\to \Eh^{k+1}_{\defin(q-1)}(X)\]
because the partial derivatives of definable functions are definable.
In the globally subanalytic case, it also induces a well-defined map
\[  d:\Eh^k_{\con(q)}(X)\to \Eh^{k+1}_{\con(q-1)}(X)\]
because the derivative of a constructible function is constructible, see the end of Section~\ref{1.1}.
\end{rem}

\section{De Rham cohomology}

\noindent
Let $p\in \IN\cup\{\infty,\omega\}$ and let $X$ be a definable $C^p$-manifold of dimension $n$. Let
$0\leq q\leq p-1$. This includes the case $p=1$ and $q=0$.

\begin{defn}\label{def:Omega-def-con-forms}
Let $q>0$ and $k\in\IN_0$.
	\begin{itemize}
		\item[(a)] We denote by $\Omega^k_{\defin(q)}(X)$ the set of all $\omega\in \Eh^k_{\defin(q)}(X)$ such that $d\omega\in\Eh^{k+1}_{\defin(q)}(X)$.
		\item[(b)] If $X$ is globally subanalytic,
we denote by $\Omega^k_{\con(q)}(X)$ the set of all $\omega\in \Eh^k_{\con(q)}(X)$ such that $d\omega\in\Eh^{k+1}_{\defin(q)}(X)$.
	\end{itemize}
\end{defn}

The case $q=0$ can be handled as well.
\begin{defnrem}
	\begin{itemize}
		\item[(1)]
Let $\omega\in \Eh^k_{\defin(0)}(X)$. Then there is a definable dense open subset $U$ of $X$ such that $\omega|_U\in \Eh^k_{\defin(1)}(U)$. 
We call $U$ a \emph{definable $C^1$-zone} for $\omega$.
\item[(2)] Assume that $X$ is globally subanalytic. Let $\omega\in \Eh^k_{\con(0)}(X)$. Then there is a globally subanalytic dense open subset $U$ of $X$ such that $\omega|_U\in \Eh^k_{\con(1)}(U)$. 
We call $U$ a \emph{globally subanalytic $C^1$-zone} for $\omega$.
	\end{itemize}
\end{defnrem}

\begin{defn}\label{def:Omega-con-forms-q=0}
	Let $q=0$ and $k\in\IN_0$.
	\begin{itemize}
		\item[(a)] We denote by $\Omega^k_{\defin(0)}(X)$ the set of all $\omega\in \Eh^k_{\defin(0)}(X)$ such that
		there is a definable $C^1$-zone $U$ for $\omega$ and some $\eta\in \Eh^{k+1}_{\defin(0)}(X)$ such that $d(\omega|_U)=\eta|_U$.
		\item[(b)] Assume that $X$ is globally subanalytic.
	 We denote by $\Omega^k_{\con(0)}(X)$ the set of all $\omega\in \Eh^k_{\con(0)}(X)$ such that
	there is a globally subanalytic $C^1$-zone $U$ for $\omega$ and some $\eta\in \Eh^{k+1}_{\con(0)}(X)$ such that $d(\omega|_U)=\eta|_U$.
	\end{itemize}
\end{defn}

\noindent
In the globally subanalytic case we will write $\Omega^k_{\sub(q)}(X)$ instead of $\Omega^k_{\defin(q)}(X)$.

\begin{rem}
	Let $k\in \IN_0$.
	\begin{itemize}
		\item[(1)] $\Omega^k_{\defin(q)}(X)$ is an $\IR$-vector subspace of $\Eh^k_{\defin(q)}(X)$. 
		\item[(2)] Assume that $X$ is globally subanalytic. Then $\Omega^k_{\con(q)}(X)$ is an $\IR$-vector subspace of $\ma{E}^k_{\con(q)}(X)$.
	\end{itemize}
\end{rem}

\begin{proof}
	The case $q>0$ being clear, we show the case $q=0$. Here we show (2), the proof for (1) being identical.
	Let $\omega_1,\omega_2\in \Omega^k_{\con(0)}(X)$ and $\lambda_1,\lambda_2\in \IR$. For $i\in \{1,2\}$ choose globally subanalytic $C^1$-zones $U_i$ for $\omega_i$ and $\eta_i\in \Eh^{k+1}_{\con(0)}(X)$ such that $d(\omega_i|_{U_i})=\eta_i|_{U_i}$.
	Then $U:=U_1\cap U_2$ is a globally subanalytic $C^1$-zone for both $\omega_1$ and $\omega_2$.
	Let $\omega:=\lambda_1\omega_1+\lambda_2\omega_2\in \Eh^k_{\con(0)}(X)$ and $\eta:=\lambda_1\eta_1+\lambda_2\eta_2\in \Eh^{k+1}_{\con(0)}(X)$.
	We have that
	$$d(\omega|_U)=\lambda_1d(\omega_1|_U)+\lambda_2d(\omega_2|_U)=\lambda_1\eta_1|_U+\lambda_2\eta_2|_U=\eta|_U$$
	and obtain that $\omega\in \Omega^k_{\con(0)}(X)$.
\end{proof}

\begin{rem}\label{R 2.5}
	Let $k\in \IN_0$.
	\begin{itemize}
		\item[(1)] We have $\Eh^k_{\defin(q+1)}(X)\subset \Omega^k_{\defin(q)}(X)$. Furthermore, $\Omega^k_{\defin(q)}(X)$ is a module over the ring $C_\defin^{q+1}(X)$ of definable $C^{q+1}$-functions on $X$.
		\item[(2)] Assume that $X$ is globally subanalytic. We have $\Eh^k_{\con(q+1)}(X)\subset \Omega^k_{\con(q)}(X)$.  Furthermore, $\Omega^k_{\con(q)}(X)$ is a module over the ring $C_\con^{q+1}(X)$ of constructible $C^{q+1}$-functions on $X$.

	\end{itemize}
\end{rem}
%\begin{proof}
	%(1): This is clear since the derivative of a differentiable definable function is definable. 
	
	%(2): In \cite[Theorem A]{kaiser_rocky} it was shown that the derivative of a differentiable function belonging to the larger class of log-analytic functions is again log-analytic. Adapting the arguments there, using thereby the preparation theorem in \cite[Section 3]{cluckers_miller} we obtain the same for constructible functions. 
%\end{proof}

\begin{rem}
	Let $q=0$ and $k\in \IN_0$.
	\begin{itemize}
		\item[(1)] Let $\omega\in \Omega^k_{\defin(0)}(X)$. Let $U,U'$ be definable $C^1$-zones for $\omega$ and let $\eta,\eta'\in \ma{E}^{k+1}_{\defin(0)}(X)$ be such that $d(\omega|_U)=\eta|_U$ and $d(\omega|_{U'})=\eta'|_{U'}$.
		Then $\eta=\eta'$.
		\item[(2)] Let $\omega\in \Omega^k_{\con(0)}(X)$. Let $U,U'$ be globally subanalytic $C^1$-zones for $\omega$ and let $\eta,\eta'\in \ma{E}^{k+1}_{\con(0)}(X)$ be such that $d(\omega|_U)=\eta|_U$ and $d(\omega|_{U'})=\eta'|_{U'}$.
		Then $\eta=\eta'$.
	\end{itemize}
\end{rem}
\begin{proof}
	Let $V:=U\cap U'$. Then $V$ is dense in $X$.
	We have that
	$\eta|_V=d(\omega|_V)=\eta'|_V$. The forms $\eta$ and $\eta'$ being continuous forms on $X$ that agree on the dense open subset $V$ of $X$, must agree everywhere, in other words $\eta = \eta'$. 
\end{proof}

\begin{rem}
	Let $q=0$ and $k\in \IN_0$.
		\begin{itemize}
		\item[(1)] Let $\omega\in \Omega^k_{\defin(0)}(X)$. Let $U$ be a definable $C^1$-zone for $\omega$ and let $\eta\in \ma{E}^{k+1}_{\defin(0)}(X)$ be such that $d(\omega|_U)=\eta|_U$. Then $\eta\in \Omega^{k+1}_{\defin(0)}(X)$. 
		\item[(2)] Let $\omega\in \Omega^k_{\con(0)}(X)$. Let $U$ be a globally subanalytic $C^1$-zons for $\omega$ and let $\eta,\in \ma{E}^{k+1}_{\con(0)}(X)$ be such that $d(\omega|_U)=\eta|_U$.
		Then $\eta\in \Omega^{k+1}_{\con(0)}(X)$. 
\end{itemize}
	\end{rem}
	\begin{proof}
	Let $\ma{A}$ be a definable $C^p$-atlas for $X$. 
    Let $U$ be a $C^1$-zone for both $\omega$ and $\eta$.
    Let $\phi:V\xrightarrow{\sim} W$ be a chart from $\ma{A}$. We show that $d\zeta=0$ where $\zeta:=\eta|_{U\cap V}$ and are done.
    It suffices to show that $\psi^*d\zeta=0$ where $\psi:=\phi^{-1}:W\to V$. 
    There is an open and dense globally subanalytic subset $W'$ of $\phi(U\cap V)$ such that $\psi^*(\omega|_V)$ is $C^2$ on $W'$.
    Let $V':=\phi^{-1}(W')$. Note that $V'$ is open and dense in $U\cap V$.
    We obtain 
    $$(\psi^*d\zeta)|_{W'}=d\psi^*(\zeta|_{V'})=d\psi^*(\eta|_{V'})=d\psi^*d(\omega|_{V'})=d^2\psi^*(\omega|_{V'})=0$$
    and are done.
   \end{proof}

\noindent By the previous remarks the following is well-defined.

\begin{defn}\label{def:exterior-derivative-Omega-q=0}
	Let $q=0$ and $k\in \IN_0$.
	\begin{itemize}
		\item[(a)] We define $D:\Omega^k_{\defin(0)}(X)\to \Omega^{k+1}_{\defin(0)}(X)$ as follows: Let $\omega\in \Omega^k_{\defin(0)}(X)$.
		Let $U$ be a definable $C^1$-zone for $\omega$ and $\eta\in \ma{E}^{k+1}_{\defin(0)}(X)$ be such that $d(\omega|_U)=\eta|_U$. Then we set $D\omega:=\eta$. 
		\item[(b)]  Assume that $X$ is globally subanalytic. We define $D:\Omega^k_{\con(0)}(X)\to \Omega^{k+1}_{\con(0)}(X)$ as follows: Let $\omega\in \Omega^k_{\con(0)}(X)$.
		Let $U$ be a globally subanalytic $C^1$-zone for $\omega$ and $\eta\in \ma{E}^{k+1}_{\con(0)}(X)$ be such that $d(\omega|_U)=\eta|_U$. Then set $D\omega:=\eta$. 
	\end{itemize}
\end{defn}

\noindent 
Note that $D\circ D=0$ by the above proof. We occasionally write also $D$ for $d$ in the case $q>0$.
Recall that $n=\dim(X)$.

\begin{rem}
Let $*\in \{\defin,\con\}$. Assume that $X$ is globally subanalytic if $*=\con$.
\begin{itemize}
	\item[(1)] Let $q=\infty$ or $q=\omega$. Then $\Omega^k_{*(q)}(X)=\Eh^k_{*(q)}(X)$ for every $k\in \IN_0$.
	\item[(2)] Let $q<\infty$. We have $\Omega^0_{*(q)}(X)=\Eh^0_{*(q+1)}(X)$ and 
	$\Omega^n_{*(q)}(X)=\Eh^n_{*(q)}(X)$.
\end{itemize}
\end{rem}
\begin{proof}
(1) and (2) in the case $q>0$ are clear. We show (2) for $q=0$ in the case $*=\con$, the case $*=\defin$ being the same.

\vs{0.2cm}\noindent
$\Omega^0_{\con(0)}(X)=\Eh^0_{\con(1)}(X)$: 

\noindent By Remark \ref{R 2.5} we have $\Eh^0_{\con(1)}(X)\subset \Omega^0_{\con(0)}(X)$.
For the other inclusion let $\omega\in \Omega^0_{\con(0)}(X)$. Then $\omega:X\to \IR$ is a continuous constructible function. By passing to a chart, we may assume that $X$ is an open globally subanalytic subset of $\IR^n$.  There is an open and dense globally subanalytic subset $U$ of $X$ such that $\omega|_U$ is $C^1$ and for each $i\in \{1,\ldots,n\}$ there is a continuous constructible function $g_i:X\to \IR$ with $\partial \omega|_U/\partial x_i=g_i$.
Let $B:=X\setminus U$. Then $B$ is a globally subanalytic subset of $X$ with $\dim(B)<n$.
We have to show that $\omega$ is differentiable at every $p\in B$ with $(\partial \omega/\partial x_i)(p)=g_i(p)$ for every $i\in \{1,\ldots,n\}$.
By the good directions lemma \cite[Theorem (4.2) in Chapter 7]{van-den-Dries-tame-topology} the set of all  $v\in \IR^n\setminus \{0\}$ such that for every $p\in B$ there is $r>0$ with
$p+tv\in U$ for every $t\in (-r,r)\setminus \{0\}$ is dense in $\IR^n$.
Hence we can find a basis of such vectors.
By applying a suitable \emph{linear} coordinate transformation we can assume that this property holds for the unit vectors. Let $p\in B$. Without restriction we can assume that $p=0$. We are done by applying the following classical result from analysis (see for example \cite[Exercise 5.16]{apostol}).
Let $f:(-1,1)\to \IR$ be a continuous function with the following properties:
	\begin{itemize}
		\item[(1)] The function $f$ is differentiable on $(-1,1)\setminus\{0\}$.
		\item[(2)] There is a continuous function $g:(-1,1)\to \IR$ such that $f'(x)=g(x)$ for all $x\in (-1,1)\setminus\{0\}$.
	\end{itemize}
	Then $f$ is differentiable in $0$ with $f'(0)=g(0)$.

\vs{0.2cm}\noindent
$\Omega^n_{\con(0)}(X)=\Eh^n_{\con(0)}(X)$:

\noindent
By definition we have $\Omega^n_{\con(0)}(X)\subset \Eh^n_{\con(0)}(X)$. Let $\omega\in \Eh^n_{\con(0)}(X)$ and let $U$ be a globally subanalytic $C^1$-zone for $\omega$. We have $d(\omega|_U)=0$. Hence $\omega\in \Omega^n_{\con(0)}(X)$. 
\end{proof}

\noindent
In the case $q<\infty$ the situation is different for $0<k<n$:

\begin{ex}
	Let $q<\infty$. We have that $\Eh^1_{\sub(q+1)}(\IR^2)\subsetneq \Omega^1_{\sub(q)}(\IR^2)\subsetneq \Eh^1_{\sub(q)}(\IR^2)$.
\end{ex}

	\begin{proof}
    $\Eh^1_{\sub(q+1)}(\IR^2)\subsetneq \Omega^1_{\sub(q)}(\IR^2)$:

	\noindent
	By Remark \ref{R 2.5} we have  $\Eh^1_{\sub(q+1)}(\IR^2)\subset \Omega^1_{\sub(q)}(\IR^2)$.
    Set $$c:\IR\to \IR,
	t\mapsto \left\{\begin{array}{ccc}
		0,&&t\leq 0,\\
		&\mbox{if}&\\
		t^{q+1},&&t\geq 0.
	\end{array}\right.$$
	Then $c$ is a globally subanalytic $C^q$-function which is not $C^{q+1}$.
	Set $\omega:=c(x_1)dx_1$.
	Then $\omega\in \ma{E}_{\sub(q)}^1(\IR^2)\setminus \ma{E}_{\sub(q+1)}^1(\IR^2)$. 
    The set $U:=\IR^2\setminus (\{0\}\times \IR)$ is a globally subanalytic $C^1$-zone for $\omega$. We have $d(\omega|_U)=0$. Hence
    $\omega\in \Omega_{\sub(q)}^1(\IR^2)$. 

	\vs{0.2cm}\noindent
	$\Omega^1_{\sub(q)}(\IR^2)\subsetneq \Eh^1_{\sub(q)}(\IR^2)$:
	
	\noindent
	By Definition we have $\Omega^1_{\sub(q)}(\IR^2)\subset \Eh^1_{\sub(q)}(\IR^2)$.
	Let $c$ be as above and set
	$\omega:=c(x_1)dx_2$. Then $\omega\in \ma{E}^1_{\sub(q)}(\IR^2)$. We have that $U:=\IR^2\setminus (\{0\}\times \IR)$ is a globally subanalytic $C^1$-zone for $\omega$ and that
	$d\omega(x)=c'(x_1)dx_1\wedge dx_2$ for $x\in U$. The function $c'(t)$ is not $C^q$ on $\R$. Hence $d(\omega|_U)$ cannot be extended to a form in $\Eh^2_{\sub(q)}(\IR^2)$ and therefore $\omega\in \ma{E}^1_{\sub(q)}(\IR^2)\setminus \Omega^1_{\sub(q)}(\IR^2)$.
	
\end{proof}

Recall that $p\in \IN\cup\{\infty,\omega\}$ and $0\leq q\leq p-1$.

\begin{defn}\label{def:de-rham-complex-definable-const}
	\begin{itemize}
		\item[(a)] We define the 
		\emph{definable de Rham complex} $(\Omega^\bullet_{\defin(q)}(X),D)$ of $X$ by
		$$0\to \Omega^0_{\defin(q)}(X)\xrightarrow{D} \Omega^1_{\defin(q)}(X)\xrightarrow{D} \ldots\xrightarrow{D} \Omega^n_{\defin(q)}(X)\to 0$$
		with the
		\emph{definable $k^\mathrm{th}$ de Rham cohomology group of $X$} as the $k^\mathrm{th}$-cohomology group of the above complex of real vector spaces, i.e. \[H^k_{\mathrm{dR},\defin(q)}(X) := \ker(\Omega^k_{\defin(q)}(X)\xrightarrow{D} \Omega^{k+1}_{\defin(q)}(X))/\mathrm{im}(\Omega^{k-1}_{\defin(q)}(X)\xrightarrow{D} \Omega^{k}_{\defin(q)}(X)).\]
		\item[(b)] In the case $\Mh=\R_\an$, we define the
		\emph{constructible de Rham complex} $(\Omega^\bullet_{\con(q)}(X),D)$ of $X$ by
$$0\to \Omega^0_{\con(q)}(X)\xrightarrow{D} \Omega^1_{\con(q)}(X)\xrightarrow{D} \ldots\xrightarrow{D} \Omega^n_{\con(q)}(X)\to 0$$
		with the
		\emph{constructible de Rham cohomology groups} $H^k_{\mathrm{dR},\con(q)}(X)$ as the $k^\mathrm{th}$-cohomology group of the above complex.
	\end{itemize}
\end{defn}

In the globally subanalytic case we write $H^\bullet_{\mathrm{dR},\sub(q)}(X)$.

For general o-minimal structures, the definable de Rham cohomology  does \emph{not}
agree with singular cohomology, as seen in the following example.

\begin{ex} 
	\label{not computing}
	Let $I=(0,1)$ be the open unit interval. We have that $H^1_{\mathrm{dR},\mathrm{sub}(q)}(I) \neq 0$ for every $q\in \IN_0\cup\{\omega\}$.
\end{ex}

\begin{proof}
	Let $\omega:=dx/x\in \Omega^1_{\mathrm{sub}(q)}(I)$.
	Then $d\omega=0$ but there is no $f\in \Omega^0_{\mathrm{sub}(q)}(I)$ with $df=\omega$.
	The $C^\omega$-function $g:I\to \IR, x\mapsto\log x,$ which fulfills $dg=\omega$ is not globally subanalytic.
\end{proof}

It will turn out that the situation is much better in the constructible setting.

\section{Basic Results}

Let $p\in \IN\cup\{\infty,\omega\}$ and $X,Y$ be definable $C^p$-manifolds and let $f:Y\to X$ be a definable $C^p$-map. Let $0\leq q\leq p-1$.

\subsection{Functoriality}

In the case $q>0$ we have the usual functoriality.

\begin{rem}\label{functoriality_easy}
Let $q>0$ and $k\in \IN_0$.
	\begin{itemize}
		\item[(1)] We have the well-defined map $\Omega^k_{\defin(q)}(X)\to \Omega^k_{\defin(q)}(Y), \omega\mapsto f^*\omega,$
		with $D(f^*\omega)=f^*(D\omega)$.
		\item[(2)] In the globally subanalytic case we have the well-defined map $\Omega^k_{\con(q)}(X)\to \Omega^k_{\con(q)}(Y), \omega\mapsto f^*\omega,$ with $D(f^*\omega)=f^*(D\omega)$.
	\end{itemize}
\end{rem}

Our aim is to establish functoriality also in the case $q=0$. We will need some preparation.

\begin{prop}\label{prop:analysis}
Let $M$ be a $C^2$-manifold with boundary.  We denote by $\iota: \partial M\to M$ the inclusion of the boundary.
Assume that
$\omega$ is a continuous $k$-form on $M$ whose restriction $\omega|_{M^\circ}$ to the interior $M^\circ$ of $M$ and $\iota^*\omega$ to the boundary $\partial M$ are $C^1$. Let, moreover, $\eta$ be a continuous $(k+1)$-form on $M$ such that
$d\omega|_{M^\circ}=\eta|_{M^\circ}$. Then 
\[ d(\iota^*\omega)=\iota^*\eta.\]
\end{prop}
\begin{proof}
The question is local on $\partial M$, hence it suffices to consider the case
$M=[0,1)\times (-1,1)^l$ and $\partial M=\{0\}\times (-1,1)^l$. We use the coordinates $x_0,\dots,x_l$. We spell out the claim. 

We have (in multi-index notation)
\[ \omega=\sum_{J\in \ma{J}}a_Jdx^J, \quad\eta=\sum_{J'\in \ma{J'}}b_{J'}dx^{J'}\]
where 
\begin{align*}
\ma{J}&:=\left\{(j_1,\ldots,j_k)\mid 0\leq j_1<\dots<j_k\leq l\right\},\\
\ma{J'}&:=\left\{(j_1,\ldots,j_{k+1})\mid 0\leq j_1\dots<j_{k+1}\leq l\right\}.
\end{align*}
By assumption for every $J\in \ma{J}$ respectively $J'\in \ma{J}'$, the functions $a_J$ respectively $b_{J'}$ are continuous and
even $C^1$ for $x_0>0$. We denote by $\omega^0$ and $\eta^0$ the restriction of $\omega$ respectively $\eta$
to $\{x_0=0\}$. Note that $\iota^*dx_0$ vanishes on $\partial M$. Hence  we have
\[ \omega^0=\sum_{J\in \tilde{\ma{J}}}a^0_J dx^J, \quad\eta^0=\sum_{J'\in \tilde{\ma{J'}}}b^0_{J'}dx^{J'}\]
where
\begin{align*}
\tilde{\ma{J}}&:=\left\{(j_1,\ldots,j_k)\mid 1\leq j_1<\dots<j_k\leq l\right\},\\
\tilde{\ma{J'}}&:=\left\{(j_1,\ldots,j_{k+1})\mid 1\leq j_1\dots<j_{k+1}\leq l\right\}.
\end{align*} 
  and 
\[ a^0_J(x_1,\dots,x_l)=a_J(0,x_1,\dots,x_l), \quad b^0_{J'}(x_1,\dots,x_l)=b_{J'}(0,x_1,\dots,x_l)\]
for $J\in \tilde{\ma{J}}$ respectively $J'\in \tilde{\ma{J'}}$.
We assume that
$d\omega=\eta$
for $x_0>0$, i.e., for all $0\leq j_1<\ldots<j_{k+1}\leq l$ and for all $(x_0,x)\in (0,1)\times (-1,1)^l$ 
\[
 b_{j_1\dots j_{k+1}}(x_0,x)=\sum_{i=1}^{k+1}(-1)^{i-1}\frac{\partial a_{j_1\dots\hat{j_i}\dots j_{k+1}}}{\partial x_{j_i}}(x_0,x)
\]
and claim
$d\omega^0=\eta^0,$
i.e., for all $1\leq j_1<\dots<j_{k+1}\leq l$ 
\[ b^0_{j_1\dots j_{k+1}}(x)=\sum_{i=1}^{k+1}(-1)^{i-1}\frac{\partial a^0_{j_1\dots\hat{j_i}\dots j_{k+1}}}{\partial x_{j_i}}(x)\]
for every $x\in (-1,1)^l$.
Without loss of generality, we consider $j_1=1,\dots,j_{k+1}=k+1$. 
Without loss of generality, it suffices
to compare the values in $0$.
We simplify notation:
\begin{align*}
 a_i^t(x_1,\dots,x_{k+1})&:=a_{1\dots\hat{i}\dots k+1}(t,x_1,\dots,x_{k+1},0,\dots,0), \\
b^t(x_1,\dots,x_{k+1})&:=b_{1\dots k+1}(t,x_1,\dots,x_{k+1},0,\dots,0).
\end{align*} 
Note that $a_i^t$ and $b^t$ are continuous as functions in $(t,x_1,\dots,x_{k+1})$
and $C^1$ for fixed $t\in [0,1)$. We introduce the vector field
\[ A^t=(a_1^t,-a_2^t,\dots,(-1)^{k}a_{k+1}^t).\]
The assumption reads
\begin{equation}\label{formula} b^t=\sum_{i=1}^{k+1}(-1)^i\frac{\partial a_i^t}{\partial x_i}=\div(A^t)\end{equation}
for $t>0$ and the claim is the same equality \eqref{formula} for $t=0$.

We apply Gauss's Theorem to the vector fields $A^t$ and the balls
$\IB(r)\subset (-1,-1)^{k+1}$  of radius $r<1$:
\begin{align*}
\div(A^0)(0)&= \lim_{r\to 0}\frac{1}{\vol(\IB(r))}\int_{\IB(r)}\div(A^0)d\IB(r)\\
         &=\lim_{r\to 0}\frac{1}{\vol(\IB(r))}\int_{\partial \IB(r)}\langle A^0(x),x\rangle d\partial \IB(r)\\
         &=\lim_{r\to 0}\lim_{t\to 0}\frac{1}{\vol(\IB(r))}\int_{\partial \IB(r)}\langle A^t(x),x\rangle d\partial \IB(r)\\
         &=\lim_{r\to 0}\lim_{t\to 0}\frac{1}{\vol(\IB(r))}\int_{\IB(r)}\div (A^t)d\IB(r)\\
    &=\lim_{r\to 0}\lim_{t\to 0}\frac{1}{\vol(\IB(r))}\int_{\IB(r)}b^td\IB(r)\\
    &=\lim_{r\to 0}\frac{1}{\vol(\IB(r))}\int_{\IB(r)}b^0d\IB(r)\\
    &=b^0(0)
\end{align*}
as claimed. 
\end{proof}

\begin{prop}\label{prop:stratification}
\begin{itemize}
	\item[(1)]
Let $\ma{D}$ be a finite set of definable subsets of $X$.
Let $\omega\in\Omega^k_{\defin(0)}(X)$ and set $\eta:=D\omega\in \Omega^{k+1}_{\defin(0)}(X)$. Then
there is a stratification of $X$, compatible with $\ma{D}$, into finitely many definable $C^1$-submanifolds $\iota_S:S\to X$ such that
$\iota_S^*\omega\in\Omega^k_{\defin(0)}(S)$  with $D(\iota_S^*\omega)=\iota_S^*\eta$.
\item[(2)] 
Assume that $X$ is globally subanalytic. Let $\ma{D}$ be a finite set of 
globally subanalytic subsets of $X$.
Let $\omega\in\Omega^k_{\con(0)}(X)$ and set $\eta:=D\omega\in \Omega^{k+1}_{\con(0)}(X)$. Then
there is a stratification of $X$, compatible with $\ma{D}$, into finitely many globally subanalytic $C^1$-submanifolds $\iota_S:S\to X$ such that
$\iota_S^*\omega\in\Omega^k_{\con(0)}(S)$  with $D(\iota_S^*\omega)=\iota_S^*\eta$.
\end{itemize}
\end{prop}
\begin{proof}
We show (2), the proof for (1) being the same.
By enlarging the set $\ma{D}$ we can assume that the domains of a given globally subanalytic $C^p$-atlas of $X$ are contained. Hence it is both sufficient and necessary to consider a chart of that atlas.
Hence we may assume that $X$ is an open globally subanalytic subset of $\IR^n$. There is a  family $\ma{S}_0$ of finitely many pairwise disjoint open globally subanalytic subsets of $X$ which are compatible with $\ma{D}$ such that $\bigcup_{S_0\in \ma{S}_0}S_0$ is a globally subanalytic $C^1$-zone for $\omega$. 

Then the complement $X\setminus \bigcup_{S_0\in \ma{S}_0}S_0$ has smaller dimension. 
It contains a dense and relatively open globally subanalytic subset $T'$ which is a $C^2$-submanifold of $X$. Let $m$ be its dimension and let $\iota:T'\to X$ be 
the inclusion. 
The form $\iota^*\omega$ is constructible and continuous.
Hence there is a family $\ma{T}$  of finitely many pairwise disjoint globally subanalytic $C^2$-submanifolds of $T'$ of dimension $m$ which are  compatible with $A$  such that $\bigcup_{T\in \ma{T}}T$ is a globally subanalytic $C^1$-zone for $\iota^*\omega$.  

By refining we may assume that for $T\in \ma{T}$ there is  $S_0\in \ma{S}_0$ such that $T$ is contained in the closure of $S_0$.
Fix $T\in \ma{T}$ and such an $S_0$. By the good direction lemma in \cite[Theorem (4.2) in Chapter 7]{van-den-Dries-tame-topology} there is
some $v\in\R^n\setminus\{0\}$ such that for every $Q\in T$ there is $r>0$ such that $Q+tv\in S_0$ for all $t\in(0,r)$. By cell decomposition we find a dense globally subanalytic set $S_T$ which is relatively open in $T$ such that there is a continuous globally subanalytic function $\alpha_T:S_T\to \IR_{>0}$ with $Q+tv\in S_0$ for all $Q\in S_T$ and $0<t<\alpha_T(Q)$.
Note that 
$$M:=\{Q+tv\mid Q\in S_T, t\in [0,\alpha_T(Q))\}$$ is a $C^2$-manifold with boundary. Set  $\ma{S}_1:=(S_T)_{T\in \ma{T}}$.

Now consider $X\setminus \bigcup_{S\in\ma{S}_0\cup\ma{S}_1} S$. Inductively, we construct a sequence of finite families of pairwise disjoint globally subanalytic $C^2$-submanifolds $\ma{S}_0,\dots,\ma{S}_N$ of $X$ such that $\ma{S}=\ma{S}_0\cup\ldots\cup \ma{S}_N$ is a stratification of $X$ compatible with $\ma{D}$ fulfilling 
the following property: For every $T\in \ma{S}$ of dimension less than $\dim X$ there is $S_0\in\ma{S}_0$ and a globally subanalytic $C^2$-submanifold $M\subset S_0$
of dimension $\dim(T)+1$ such that $T\cup M$ is a $C^2$-manifold with boundary.
Moreover, $\iota_X^*\omega$ is $C^1$ for all $S\in\ma{S}$ where $\iota_S:S\to X$ is the inclusion.

Let $S\in \ma{S}$ and let $\iota_S:S\to X$ be the inclusion. We show that $\iota^*_S\omega\in \ma{E}_{\mathrm{con}(1)}^k(S)$ with $d(\iota_S^*\omega)=\iota_S^*\eta$. The case $S\in \ma{S}_0$ is is clear since then $S$ is an open subset of a $C^1$-zone of $\omega$. The case $S\in \ma{S}_j$ for some $j>0$ follows from the above and   Proposition~\ref{prop:analysis}. 
\end{proof}

\begin{thm}\label{thm:functorial}
Let $q=0$ and $k\in \IN_0$. Let $f:Y\to X$ be a definable $C^1$-map of definable $C^1$-manifolds.
\begin{itemize}
	\item[(1)] We have the well-defined map $\Omega^k_{\defin(0)}(X)\to \Omega^k_{\defin(0)}(Y), \omega\mapsto f^*\omega,$
	with $D(f^*\omega)=f^*(D\omega)$.
	\item[(2)] In the case $\ma{M}=\IR_\an$ we have the well-defined map $\Omega^k_{\con(0)}(X)\to \Omega^k_{\con(0)}(Y), \omega\mapsto f^*\omega,$ with $D(f^*\omega)=f^*(D\omega)$.
\end{itemize}
\end{thm}
\begin{proof}
We show again (2).
Let $\omega\in\Omega^k_{\con(0)}(X)$ and $\eta:=D\omega\in \Omega^{k+1}_{\con(0)}(X)$. Let $U$ be a globally subanalytic $C^1$-zone for $\omega$ such that $d(\omega|U)=\eta|_U$.
We have $f^*\omega\in \ma{E}^k_{\con(0)}(Y)$. 
(This is where we use that $f$ is at least $C^1$.)
We need to check that
$d((f^*\omega)|_V)=(f^*\eta)|_V$ on some $C^1$-zone $V$ for $f^*\omega$. The assertion
is local on $X$ and $Y$. We pass to charts and replace $Y$ by a globally subanalytic dense open subset on which $f$ is even $C^2$.

Without loss of generality $f^*\omega$ is $C^1$ (replace $Y$ by a $C^1$-zone for $f^*\omega$). The image $f(Y)\subset X$ is globally subanalytic. Let $\ma{S}$ be a stratification of $X$, compatible with $f(Y)$, as in Proposition~\ref{prop:stratification}. Hence there is
a dense open globally subanalytic $\iota:X'\hookrightarrow f(Y)$ which is
also a  $C^1$-submanifold of $X$ and such that $\iota^*\omega$ is $C^1$. Moreover,
$d(\iota^*\omega)=\iota^*\eta$.
The preimage $f^{-1}(X')$ is dense
in $Y$. We claim that this is the $C^1$-zone we wanted to find.
Note that 
$f^*\omega|_{f^{-1}(X')}=f^*\iota^*\omega$.
The claim now follows from the compatibility of
pull-back and differential in the $C^1$-case. 
\end{proof}

\subsection{Stokes' Theorem}

In the case $q>0$ we immediately have the usual Stokes' theorem for simplices.

  \begin{rem}\label{Stokes}
  	Let $q>0$.
   	Let $k>0$ and let $\Delta^k$ be the standard $k$-simplex.
   	Let  $f:\Delta^k\to X$ be a definable $C^1$-map.
   	Then Stokes' formula holds: 
   	\begin{itemize}
   		\item[(1)]
   	Let $\omega\in\Omega_{\defin(q)}^{k-1}(X)$. Then 
   	$$\int_{\Delta^k}f^*d\omega=\int_{\partial \Delta^k}f^*\omega.$$
   	\item[(2)]
   	In the globally subanalytic case 
   	let $\omega\in\Omega_{\con(q)}^{k-1}(X)$. Then 
   	$$\int_{\Delta^k}f^*d\omega=\int_{\partial \Delta^k}f^*\omega.$$
   \end{itemize}
   \end{rem}

   We can establish Stokes' theorem also in the case $q=0$. 
   
    \begin{prop}\label{Stokes +}
   	Let $q=0$.
   	Let $k>0$ and let $\Delta^k$ be the standard $k$-simplex.
   	Let  $f:\Delta^k\to X$ be a definable $C^1$-map.
   	Then Stokes' formula holds: 
   	\begin{itemize}
   		\item[(1)]
   		Let $\omega\in\Omega_{\defin(0)}^{k-1}(X)$. Then 
   		$$\int_{\Delta^k}f^*D\omega=\int_{\partial \Delta^k}f^*\omega.$$
   		\item[(2)]
   		In the globally subanalytic case 
   		let $\omega\in\Omega_{\con(0)}^{k-1}(X)$. Then 
   		$$\int_{\Delta^k}f^*D\omega=\int_{\partial \Delta^k}f^*\omega.$$
   	\end{itemize} 
   \end{prop}
   \begin{proof}
   	We show (2), the proof being identical for (1).
   	We fix $\omega$. Let $\eta\in\Omega^k_{\con(0)}(X)$ be such that
   	$\eta|_U=d(\omega|_U)$ on a globally subanalytic $C^1$-zone $U$ for $\omega$.
   	By functoriality (Theorem \ref{thm:functorial}) it suffices to consider the case where $f$ is the inclusion of $\Delta^k$ in an open neighbourhood of $\Delta^k$ in $\R^k$ . By \cite[Main Theorem]{omin-triang} there is a globally subanalytic triangulation of $\Delta^k$ by $C^1$-simplices such that $\omega$ is $C^1$ on the interior of each face. It suffices to establish the formula for these. Without loss of generality, $\omega$ is now $C^1$ on the interior of $\Delta^k$ and its faces.

   	Let $\Delta_\epsilon\subset\Delta^k$ be the $k$-simplex at distance $\epsilon$ from the boundary. It is fully contained in the $C^1$-zone for $\omega$. Hence Stokes' formula holds on $\Delta_\epsilon$. We take the limit $\epsilon\to 0$.
   	As $\omega$ and $\eta$ are continuous, we obtain the terms for $\Delta^k$ itself.
   \end{proof}

   \section{Homotopy invariance}
   
   By Example \ref{not computing} homotopy invariance does in general not hold for the definable de Rham cohomology groups. 
   But we are able to establish it in the globally subanalytic setting for the constructible de Rham cohomology groups.
   
   Let $X,Y$ be globally subanalytic $C^p$-manifolds where $p\in \IN\cup\{\omega\}$ and let $f,g:X\to Y$ be globally subanalytic $C^p$-mappings. Let $0\leq q\leq p-1$.
   
   We show that the constructible de Rham cohomology is invariant under globally subanalytic $C^p$-homotopies.  We establish in the first step the case $q>0$. The classical proof  (see for example \cite[Chapter I \S 4]{bott_tu}) has to be adjusted to our setting.
   In the second step we will show also the case $q=0$.
   
   We consider $X\times \R$. Let $\pi:X\times\R\to X, (x,t)\mapsto x,$ be the natural projection and
   $s:X\to X\times \R, x\mapsto (x,0),$ the $0$-section.
The natural decomposition
   \[ 
\Lambda^kT^*(X\times\R)=\Lambda^kT^*X\oplus \Lambda^{k-1}T^*X\times T^*\R
\]
   induces a canonical decomposition
  \begin{equation}\label{can} 
    \omega=\omega'+\omega''\wedge dt
\end{equation}
   for all $\omega\in\Omega^k_{\con(q)}(X\times\R)$. In local coordinates,
   $\omega''\wedge dt$ collects the components involving $dt$ and $\omega'$ the forms
   that do not involve $dt$. Note that both $\omega'$ and $\omega''$ are constructible and $C^q$. However, they are not necessarily in $\Omega^k_{\con(q)}(X\times\R)$ because the differential does not respect the decomposition. Indeed,
   when we decompose
   \[ d=d_x+d_t\]
   with $d_x$ induced from the differential on $X$ and $d_t$ from the differential on $\R$, then
   \[ d\omega=d_x\omega'+d_t\omega'+d_x\omega''\wedge dt\]
   and $(d\omega)'=d_x\omega'$. If $d\omega$ is $C^q$, then so is $d_x\omega'$ but not necessarily $d_t\omega'$ and $d_x\omega''$.
   
   \begin{defn}
   	\label{DR 4.12}
   	Let $q\geq 0$.
   	Let $k\in \IN$. We set
   	\begin{align*}
   		Q_k:\Omega^k_{\con(q)}(X\times \R)&\to \Omega^{k-1}_{\con(q)}(X\times \R),\\
   		\omega&\mapsto \int_0^t\omega''dt.
   	\end{align*}
   \end{defn}

   \begin{prop}\label{homotopy_formula}
   	Let $q>0$.
   	The form $Q_k(\omega)$ is well-defined in $\Omega^{k-1}_{\con(q)}(X\times \IR)$ and satisfies
   	\[ DQ_k(\omega)-Q_{k+1}(D\omega)=(-1)^k(\omega-\pi^*s^*\omega).\]
   \end{prop}
   \begin{proof}By Cluckers and Dan Miller \cite[Theorem 1.3]{cluckers_miller} parametric integrals of integrable constructible functions are constructible.  This makes $Q_k(\omega)$ constructible. Once we have established the formula,
   	we can argue by descending induction on $k$ that $DQ_k(\omega)$ is $C^q$, so that
   	$Q_k(\omega)$ is indeed in $\Omega^{k-1}_{\con(q)}(X\times\R)$.
   	
   	The formula itself is classical, see \cite[Chapter~I, \S~4]{bott_tu}. We sketch the argument for the convenience of the reader.
   	We show the formula for $\omega'$ and $\omega''\wedge dt$ separately.
   	The claims are
   	\begin{gather}
   		0-Q_{k+1}(d_t\omega')=(-1)^k(\omega'-\pi^*s^*\omega'),\\
   		d_tQ_k(\omega''\wedge dt)=(-1)^k\omega''\wedge dt,\\
   		d_xQ_k(\omega'')-Q_{k+1}(d_x\omega''\wedge dt)=0.
   	\end{gather}
   	The first two hold by the fundamental theorem of calculus applied to the components of $\omega'$ and $\omega''$. The last holds by the formula for the derivative of a parameter dependent integral.
   \end{proof}

\subsection{The case $q=0$}
The aim of this section is to extend the formula for the chain homotopy also to the case $q=0$.

    \begin{prop}\label{homotopy_formula +}
   	Let $q=0$.
   	The form $Q_k(\omega)$ is well-defined in $\Omega^{k-1}_{\con(0)}(X\times \IR)$ and satisfies
   	\[ DQ_k(\omega)-Q_{k+1}(D\omega)=(-1)^k(\omega-\pi^*s^*\omega).\]
   \end{prop}

The proof will be given at the end of the section by combining some geometric preparations with a limit argument.

Throughout, $X$ is a globally subanalytic $C^1$-manifold. We fix $\omega\in\Omega^k_{\con(0)}(X\times\R)$.

\begin{lemma}\label{lem:pair}
Assume that $X$ is affine. Let $f:X\times\R\to\R$ be a constructible function. Then there is a dense open globally analytic subset $U\subset X$, a partition
$\ma{C}$ of $U$ into finitely many open globally subanalytic sets such that for each $C\in \ma{C}$ there are a natural number $n=n_C$ and globally subanalytic $C^2$-functions
	$c_1,\ldots,c_n:C\to \IR$ with $c_1<\ldots<c_n$ such that 
\[ \bigcup_{C\in\ma{C}}\left\{(x,t)\in C\times\R\mid t\neq c_i(x)\mbox{ for all } i\right\}\]
is a $C^1$-zone for $f$. 
\end{lemma}

\begin{proof}We write $\ma{U}=X\times\R$. Let $m=\dim X$.
	Let $\ma{V}$ be a globally subanalytic $C^1$-zone for $f$ and set
	$\ma{B}:=\ma{U}\setminus \ma{V}$. Then $\dim(\ma{B})<m+1$.
	Set $V:=\pi(\ma{V})$ and for $x\in V$ let
	$$B_x:=\{t\in \IR\mid (x,t)\in \ma{B}\}.$$
	Let $E$ be the set of all $x\in V$ such that $B_x$ is not finite, equivalently that $B_x$ contains a nonempty open interval.
	Then $\dim(E)<m$ since otherwise the interior of $B$ would be nonempty.
	By the uniform finiteness property \cite[Lemma (2.13) in Chapter 3]{van-den-Dries-tame-topology} there is $N\in \IN$ such that $\# B_x\leq N$ for all $x\in V\setminus E$. There is a cell decomposition $\ma{D}$ of $V\setminus E$ such that for every $C\in \ma{D}$ there are $n=n_C\in \{0,\ldots,N\}$ and  globally subanalytic $C^2$-functions  $c_1,\ldots,c_n:C\to \IR$ with $c_1<c_2\ldots<c_n$ such that 
	$$B_x=\{c_j(x)\mid j\in \{1,\ldots n\}\}$$
	for every $x\in C$.
	Set $\ma{C}:=\{C\in \ma{D}\mid C\mbox{ open}\}$ and set $U:=\bigcup_{C\in \ma{C}}C$. 
\end{proof}

A function $a:X\to \R$ gives rise to a section $X\to X\times \R$ of $\pi:X\times \R\to \R$. We denote this section by $s_a$. In particular $s=s_0$ is the $0$-section.

\begin{lemma} \label{lem:reduction}
In order to prove Proposition~\ref{homotopy_formula +}, it suffices to verify the following claim:

Assume $X$ is an affine $C^2$-manifold, $c_1,\dots,c_n:X\to \R$ bounded globally analytic $C^2$-functions  (either non-vanishing or constant and equal to $0$) with $c_1<\dots<c_n$ such that
$s_{c_i}^*\omega$ is $C^1$ for all $i$ and such that the complement $\ma{V}$ of the graphs is a $C^1$-zone for $\omega$. Then there is a dense globally subanalytic $U\subset X$ such that $\ma{W}:=\ma{V}\cap U\times\R$ is a $C^1$-zone for
$Q_k(\omega)$ and for all $(x,t)\in\ma{W}$
\[ d Q_k(\omega)=Q_{k+1}(d\omega)+(-1)^k(\omega-\pi^*s^*\omega).\]
\end{lemma}
\begin{proof}
As in the proof of Proposition~\ref{homotopy_formula}, the form $Q_k(\omega)$ is constructible. By descending induction on $k$ and functoriality, we may assume that
\[ E(\omega):=Q_{k+1}(D\omega)+(-1)^k(\omega-\pi^*s^*\omega)\in \Omega^{k}_{\con(0)}(X\times\R).\]
Once we have established 
\begin{equation}\label{eq:E} d Q_k(\omega)= E(\omega)\end{equation}
on a $C^1$-zone for $\omega$, this shows that $Q_k(\omega)\in\Omega^{k-1}_{\con(0)}(X\times\R)$ with $DQ_k(\omega)=E(\omega)$ as claimed.

The claim \eqref{eq:E} is local on $X$, hence we may assume that $X$ is affine. We apply Lemma~\ref{lem:pair} to the component functions of $\omega$. It suffices to consider only the open strata $C\subset X$ one at the time. 

Again as the claim is local on $X$, it suffices for each $x\in X$ to consider a gobally subanalytic neighbourhood on which all $c_i$ are bounded. (We may need infinitely many such open sets but this does not matter to the reduction.) We replace $X$ by such a neighbourhood. Finally, the vanishing locus of each $c_i$ is a globally subanalytic subset. Up to sets of smaller dimension, we may decompose $X$ into open globally subanalytic subsets on which $c_i$ does not vanish or is constant and equal to $0$. It suffices to consider each of the pieces separately.
\end{proof}

From now on, we consider $X$ as Lemma~\ref{lem:reduction}. 
We decompose $\omega=\omega'+\omega''\wedge dt$ as in \eqref{can}. 

   \begin{defn} 
   	Let $k\in \IN$. Let $a,b:X\to \R$ with $a<b$ be globally subanalytic $C^2$-functions equal to one of the $c_i$ or disjoint from all of them. We set
   	\begin{align*}
   		Q_{a}^{b}:\Omega^k_{\con(0)}(X\times \R)&\to \Omega^{k-1}_{\con(0)}(X),\\
   		\omega&\mapsto \int_{a(x)}^{b(x)}\omega''dt.
   	\end{align*}
   \end{defn}

\begin{lemma}\label{lemma:section} The form $Q_a^b(\omega)$ is well-defined in
$\Omega^{k-1}_{\con(0)}(X)$.
There is a dense open globally subanalytic $U\subset X$ such that on $U$
the form $Q_{a}^{b}(\omega)$ is $C^1$ with
\[ d Q_{a}^{b}(\omega)-Q_{a}^{b}(d\omega) = (-1)^k(s_a^*\omega-s_b^*\omega)\]
\end{lemma}
\begin{proof}  Once we have established the formula for $dQ_a^b(\omega)$ on the $C^1$-zone $U$, we put
\[ DQ_a^b(\omega):=Q_a^b(\omega)+(-1)^k(s_a^*\omega-s_b^*\omega).\]
It is continuous on $X$. This will make $Q_a^b$ well-defined.
 
We now establish the formula. We consider the strip
\[ S=\{(x,t)| x\in X, a(x)\leq t\leq b(x)\}.\]
If $\omega$ is $C^1$ on a neighbourhood of $S$, the formula follows from Proposition~\ref{homotopy_formula} by pull-back along sections (or by repeating the computation). We want to reduce to (a limit of) this case.

By assumption, there are indices $i\leq j$ such that
$c_i\leq a<c_{i+1}$, $c_j\leq b\leq c_{j+1}$ (interpreting $c_{-1}=-\infty$, $c_{n+1}=\infty$ for the boundary cases). Note that
\[ Q_a^{b}=Q_a^{c_{i+1}}+Q_{c_{i+1}}^{c_{i+2}}+\dots +Q_{c_j}^{b}.\] 
It suffices to consider the summands separately, so that $\omega$ is $C^1$ in the interior of $S$.

By decomposing 
\[ Q_a^b=Q_a^{(a+b)2}+Q_{(a+b)/2}^b\]
we can even assume that $\omega$ extends to a $C^1$-form on the lower or upper boundary. Without loss generality, it extends to a $C^1$-form on a neighbourhood of the graph of $a$. (The other case is symmetric.)
By assumption both
$\omega$ and $d\omega$ extend continuously to the graph of $b$. We want to apply
Proposition~\ref{prop:analysis} to the $C^2$-manifold with boundary $S$ and
 the continuous constructible $(k-1)$-form
\[ \Omega:(x,t)\mapsto Q_{a}^{t}(\omega)\ .\]
It is $C^1$ on the interior of $S$ because $\omega$ is, with derivative
\begin{equation}\label{eq:Omega}
	 d\Omega=Q_a^t(d\omega)+(-1)^k(\omega-\pi^*s_a^*\omega).
\end{equation}
Note that the right hand side extends continuously to the graph of $b$.

The form $s_b^*\Omega=Q_a^b(\omega)$ is continuous and constructible. Hence there is a dense open globally subanalytic $U\subset X$ such that $s_b^*\Omega$ is $C^1$. On $U$ we have by  Proposition~\ref{prop:analysis} and \eqref{eq:Omega}
\[ 
d \Omega_a^b(\omega)=db^*\Omega =s_b^*\left(Q_a^t(d\omega)+(-1)^k(\omega-\pi^*s_a^*\omega)\right)=Q_a^b(d\omega)+(-1)^k(s_b^*\omega-s_a^*\omega).
\]
This is the claim.
\end{proof}

\begin{proof}[Proof of Proposition~\ref{homotopy_formula +}.]
By Lemma~\ref{lem:reduction}, it suffices to consider the case that $X$ is affine and we are given bounded globally subanalytic $C^2$-functions
$c_1,\dots,c_n:X\to X\times \R$ such that $\omega$ is $C^1$ outside of the graph of the $c_i$. They are either non-vanishing or constant equal to $0$.
By Lemma~\ref{lemma:section} there is a dense open globally subanalytic $U\subset X$ such that all $Q_{c_i}^{c_{i+1}}(\omega)$ are $C^1$ with the correct derivative. We replace $X$ by $U$ in order to simplify the notation.

We need to verify the formula for the derivative in  the $C^1$-zone. Consider $(x,t)$ in the complement of the graphs. There is a unique index $i$ such $c_i(x)<t<c_{i+1}(x)$. There is also $j$ such that
$c_j(x)\leq 0<c_{j+1}(x)$. Assume $j<i$ for simplicity. (The other cases are analogous.) In a neighbourhood of $(x,t)$ choose $T\in\R $ such that
$c_j<T<t$. Note that $\omega$ is $C^1$ on a neighbourhood of $X\times T$. We have
\[ Q_k(\omega)=\pi^*Q_0^{c_{j+1}}(\omega)+\dots+\pi^*Q_{c_{i-1}}^{c_i}(\omega)+\pi^*Q_{c_i}^{T}(\omega)+Q_{T}^t.\]
All summands are $C^1$, hence so is $Q_k(\omega)$. The derivative is
computed termwise. By Lemma~\ref{lemma:section} and Proposition~\ref{homotopy_formula} (for the last summand) this yields
\begin{align*}
 dQ_k(\omega)&=\pi^*dQ_0^{c_{j+1}}(\omega)+\dots+\pi^*dQ_{c_{i-1}}^{c_i}(\omega)+\pi^*Q_{c_i}^{T}(\omega)+dQ_{T}^t\\
     &=Q_0^t(d\omega)+(-1)^k(\omega-\pi^*s^*\omega)\\
     &=Q_{k+1}(d\omega)+(-1)^k(\omega-\pi^*s^*\omega).
\end{align*}
\end{proof}
 
\subsection{Consequences for cohomology}
 
Let $0\leq q\leq p-1$ and $X$ a globally subanalytic $C^p$-manifold.

   \begin{prop}
   	\label{P 4.15}
   	Let $\pi:X\times \IR\to X, (x,t)\mapsto x$ be the natural projection and $s:X\to X\times \IR, x\mapsto (x,0)$ the $0$-section.
   	Then the maps 
   	$$\pi^*:H^\bullet_{\mathrm{dR},\con(q)}(X)\to H^\bullet_{\mathrm{dR},\con(q)}(X\times \IR)$$
   	and
   	$$s^*:H^\bullet_{\mathrm{dR},\con(q)}(X\times \IR)\to H^\bullet_{\mathrm{dR},\con(q)}(X)$$
   	are inverse to each other.
   \end{prop}
   \begin{proof} By functoriality $s^*\circ \pi^*=
   	\mathrm{id}_{H^\bullet_{\mathrm{dR},\mathrm{con}(q)}(X)}$.
   	It remains to show that $\pi^*\circ s^*=\mathrm{id}_{H^\bullet_{\mathrm{dR},\con(q)}(X\times\R)}$. 
   	Let $\omega\in \Omega^k_{\con(q)}(X\times \IR)$  be with $D\omega=0$. By Proposition \ref{homotopy_formula} and Proposition \ref{homotopy_formula +} we have that
   	\[ \omega-\pi^*s^*\omega=(-1)^{k-1}D(Q_k\omega)\]
   	Hence $[\omega]=\pi^*s^*[\omega]$ and we are done.
   \end{proof}
   
   \begin{cor}
   	\label{C 4.16} 
   	Assume that $f$ and $g$ are
   	globally subanalytic $C^p$-homotopic.
   	Then 
   	$$f^*:H^\bullet_{\mathrm{dR},\con(q)}(Y)\mapsto H^\bullet_{\mathrm{dR},\con(q)}(X)$$
   	and
   	$$g^*:H^\bullet_{\mathrm{dR},\con(q)}(Y)\mapsto H^\bullet_{\mathrm{dR},\con(q)}(X)$$
   	coincide.
   \end{cor}
   \begin{proof}
   	Let $\varepsilon>0$ and $h:X\times (-\varepsilon,1+\varepsilon)\to Y$ be a globally subanalytic $C^p$-map such that $f=h_0$ and $g=h_1$.
   	Choose a globally subanalytic $C^p$-map $\tau:\IR\to (-\varepsilon,1+\varepsilon)$ such that $\tau(0)=0$ and $\tau(1)=1$.
   	Set 
   	$$H:X\times \IR\to Y, (x,t)\mapsto h(x,\tau(t)).$$
   	We have that
   	$f=H\circ s_0$ and $g=H\circ s_1$ and therefore
   	$f^*=s_0^*\circ H^*$ and $g^*=s_1^*\circ H^*$.
   	The maps $s_0^*$ and $s_1^*$ are both inverse to $\pi^*$ by Proposition \ref{P 4.15} and therefore equal. 
   	Hence $f^*=g^*$.
   \end{proof}
   
   \begin{cor}
   	\label{C 4.17}
   	Let $X,Y$ be globally subanalytically $C^p$-homotopy equivalent. Then $H^\bullet_{\mathrm{dR},\con(q)}(X)$ and $H^\bullet_{\mathrm{dR},\con(q)}(Y)$ are isomorphic.
   \end{cor}
   
   \begin{cor}\label{cor:contractible}
   	Let $X$ be globally subanalytically $C^p$-contractible.
   	Then
   	\[ H^k_{\mathrm{dR},\con(q)}(X)=
   	\begin{cases} \R,&k=0,\\
   		0,&k>0.
   	\end{cases} \]
   \end{cor}
   \begin{proof}
   	This follows from Corollary \ref{C 4.17} and
   	the fact that the de Rham complex for a point is concentrated in
   	degree $0$.
   \end{proof}

		\section{Geometric Preparations}
		
		Let $p\in \IN\cup\{\omega\}$.
		A globally subanalytic $C^p$-cell is a globally subanalytic $C^p$-manifold.

		\begin{rem}
			\label{R 5.1}
			Let $C\subset \IR^N$ be a globally subanalytic $C^p$-cell of dimension $n$.
			Then $C$ is globally subanalytically $C^p$-isomorphic to the $n$-dimensional open hypercube $I^n=(0,1)^n$.
		\end{rem}

		\begin{prop}
			\label{P 5.2}
			Let $C\subset \IR^M$ be a globally subanalytic $C^p$-cell. Then $C$ is globally subanalytically $C^p$-contractible.
		\end{prop}
		
		\begin{proof}
			By Remark \ref{R 5.1} we can assume that $C=I^n$ where $n$ is the dimension of $C$. Consider the polynomial $\sigma(t)=3t^2-2t^3$. We have that $\sigma(0)=0, \sigma(1)=1$ and $\sigma(t)\in [0,1]$ for all $t \in [-1/2,3/2]$.
			Let $a\in I^n$. We see that $\mathrm{id}_{I^n}$ and $c_a$ are globally subanalytically $C^p$-homotopic via
			$$h:I^n\times (-1/2,3/2)\to I^n, (x,t)\mapsto (1-\sigma(t))x+\sigma(t)a.$$
		\end{proof}	
		
		\noindent For the proof of the constructible de Rham theorem we need the following weakening of open globally subanalytic cells.

		\begin{defn}
			A subset $V$ of $\IR^n$ is called a \emph{ribbon} if there is an open globally subanalytic subset $W$ of $\IR^{n-1}$ and globally subanalytic $C^0$-functions $a,b:W\to \IR$ with $a<b$ such that 
			$$V=\{x=(x',x_n)\in \IR^n\mid x'\in W, a(x')<x_n<b(x')\}.$$
			The set $W$ is called the \emph{base} of $V$.
		\end{defn}
		
		\begin{rem}
			A ribbon in $\IR^n$ is an open globally subanaytic subset of $\IR^n$.
		\end{rem}
		
		\begin{ex}
			\label{E 5.5}
			Let $C\subset \IR^n$ be an open globally subanalytic $C^0$-cell. Then $C$ is a ribbon.
		\end{ex}
		
		\noindent 
		The intersection of open cells is in general not an open cell. But the class of ribbons is stable under finite intersections.
		
		\begin{lemma}
			\label{L 5.6}
			Let $V_1,V_2\subset \IR^n$ be ribbons. Then $V_1\cap V_2$ is a ribbon.
		\end{lemma}
		
		\begin{proof}
			Let $W_1$ be the base of $V_1$ and $W_2$ be the base of $V_2$.
			Let $a_1,b_1:W_1\to \IR$ be globally subanalytic $C^0$-functions such that
			$$V_1=\{x=(x',x_n)\in \IR^n\mid x'\in W_1, a_1(x')<x_n<b_1(x')\}$$
			and let $a_2,b_2:W_2\to \IR$ be globally subanalytic $C^0$-functions such that
			$$V_2=\{x=(x',x_n)\in \IR^n\mid x'\in W_2, a_2(x')<x_n<b_2(x')\}.$$
			Let $\tilde{W}:=W_1\cap W_2$ and set
			$$a:\tilde{W}\to \IR, x'\mapsto \max\{a_1(x'),a_2(x')\},$$
			$$b:\tilde{W}\to \IR, x'\mapsto \min\{b_1(x'),b_2(x')\}.$$
			Let 
			$$W:=\{x'\in \tilde{W}\mid a(x')<b(x')\}$$
			and
			$$V:=\{x=(x',x_n)\in \IR^n\mid x'\in W, a(x')<x_n<b(x')\}.$$
			We have that $V$ is a ribbon and $V_1\cap V_2=V$.
		\end{proof}

		\begin{prop}\label{P 5.7}
			Let $V\subset \IR^n$ be a ribbon and let $W\subset \IR^{n-1}$ be its base.
			Then the projection $V\to W$ is a globally subanalytic $C^p$-homotopy equivalence.
		\end{prop}
		
		\begin{proof}
			Let $a,b:W\to \IR$ be globally subanalytic $C^0$-functions with $a<b$ such that
			$$V=\{x=(x',x_n)\in \IR^n\mid x'\in W, a(x')<x_n<b(x')\}.$$
			By A. Valette and G. Valette \cite[Theorem 1.1]{VV_approximations} there is a globally subanalytic $C^p$-function $c:W\to \IR$ such that $a<c<b$,
			Consider 
			$$\varphi:V\to W, x=(x',x_n)\mapsto x',\;\;\;\; \psi: W\to V, x'\mapsto (x',c(x')).$$
			We have that 
			$$f:=\psi\circ\varphi: V\to V, x=(x',x_n)\mapsto (x',c(x')),\;\;\; g:=\varphi\circ\psi: W\to W, x'\mapsto x'.$$
			So $g=\mathrm{id}_W$. It remains to show that
			$f$ is globally subanalytic $C^p$-homotopic to $\mathrm{id}_V$.
			Consider as in the proof of Proposition \ref{P 5.2} the polynomial $\sigma(t)=3t^2-2t^3$
			and take
			$$h:V\times (-1/2,3/2)\to V, (x,t)=(x',x_n,t)\mapsto (x',(1-\sigma(t))x_n+\sigma(t)c(x')).$$
		\end{proof}
		
		We can weaken the regularity condition stated in Remark \ref{P 5.2}.
		
		\begin{cor}
			\label{C 5.8}
			Let $C\subset \IR^n$ be an open globally subanalytic $C^0$-cell. Then $C$ is globally subanalytic $C^p$-homotopy equivalent to a singleton.
		\end{cor}
		\begin{proof}
			
			We do induction on $n$.
			
			\vs{0.2cm}\noindent
			$n=1$: Then $C$ is an open interval and hence an open globally subanalytic $C^p$-cell.
			We are done by Proposition \ref{P 5.2}.
			
			\vs{0.2cm}\noindent
			$n-1\to n$: Let $B$ be the base of $C$. Then $B$ is an open globally subanalytic $C^0$-cell in $\IR^{n-1}$. By Example \ref{E 5.5} and Proposition \ref{P 5.7} we have that $C$ and $B$ are globally subanalytic $C^p$-homotopy equivalent. We are done by the inductive hypothesis.
		\end{proof}
	
We can formulate now the 
Poincar\'e Lemma in our setting: 

\begin{cor}
[Poincar\'e Lemma]
\label{Poincare}

	Let $C\subset \IR^n$ be an open globally subanalytic $C^0$-cell.
Then
\[ H^k_{\mathrm{dR},\con(q)}(C)=
\begin{cases} \R,&k=0,\\
	0,&k>0.
\end{cases} \]
for every $q\in \IN_0\cup \{\omega\}$.
\end{cor}
\begin{proof}
This follows from Corollary \ref{cor:contractible} and
Corollary \ref{C 5.8}.
\end{proof}

		\section{Constructible de Rham Theorem}
		
		To establish the constructible de Rham theorem we follow the classical approach by induction on open sets, using a constructible version of the Mayer-Vietoris sequence (compare with \cite{bott_tu, madsen_tornehave}). This needs partition of unity. The argument uses partition of unity, see Proposition~\ref{partition-of-unity}. 
	
	Let $p\in \IN\cup\{\omega\}$ and let $X,Y$ be globally subanalytic $C^p$-manifolds.

	We consider the globally subanalytic singular homology groups  $H^{\sing,\sub}_\bullet(X,\IR)$ and the globally subanalytic singular cohomology groups  $H^\bullet_{\sing,\sub}(X):=\mathrm{Hom}(H_\bullet^{\sing,\sub}(X),\IR)$,  
	with respect to globally subanalytic $C^1$-simplices and coefficients in the reals. Note that as for example observed in \cite[{Corollary~5.3}]{huber_period_iso_tame}  the canonical maps 
$H^{\sing,\sub}_\bullet(X,\R)\to H_\bullet^\sing(X,\R)$ and $H_{\sing,\sub}^\bullet(X)\leftarrow H^\bullet_\sing(X,\R)$  to the classical singular homology and cohomology, respectively,  are isomorphisms as a consequence of the result of Paw\l ucki \cite{pawlucki-triang}.

	\begin{defn}\label{defn:de Rham theorem}

		We say that the \emph{constructible de Rham theorem} holds for $X$ if 
		$$\langle\cdot,\cdot\rangle:  H^\bullet_{\mathrm{dR},\mathrm{con}(q)}(X)\times H_\bullet^{\sing,\sub}(X,\R)\to \IR, \quad ([\omega],[\sigma])\mapsto \int_\sigma \omega,$$
		is a perfect pairing for every $0\leq q\leq p-1$.
\end{defn}
	
	Note that the above pairing is well-defined since Stokes's theorem holds for globally subanalytic $C^1$-simplices in the constructible setting by Remark \ref{Stokes} and Proposition \ref{Stokes +}. Note also that it is functorial for globally subanalytic $C^p$-morphisms by the change of variables formula and Remark~\ref{functoriality_easy} ($q>0$) and Theorem~\ref{thm:functorial} ($q=0$).
This means that for every globally subanalytic $C^p$-morphism $f:Y\to X$ and
 $\omega_X\in H^\bullet_{\dR,\con(q)}(X)$, $\sigma_Y\in H_\bullet^{\sing,\sub}(X,\R)$ we have
\[ \langle f^*\omega_X,\sigma_Y\rangle =\langle \omega_X,f_*\sigma_Y\rangle.\]

	\begin{rem}
		\label{R 6.2}
		Let $X$ and $Y$ be globally subanalytic $C^p$-homotopy equivalent. Then the constructible de Rham theorem holds for $X$  if and only if it holds for $Y$.
	\end{rem}
	
		Our main result will be that the constructible de Rham theorem holds in the case $p<\infty$ in full generality. We will use partition of unity.
	
		\begin{lemma}
		\label{L 6.3}
		Let $p<\infty$.
		Let $U_1,U_2\subset X$ be open globally subanalytic subsets. Assume that the constructible de Rham theorem holds for $U_1,U_2$ and $U_1\cap U_2$.
		Then it holds for $U_1\cup U_2$.
	\end{lemma}
	
	\begin{proof}
		Let $0\leq q \leq p-1$. We consider for $k\in \IN_0$ the Mayer-Vietoris sequence
		\begin{equation}\label{MV}
0\mapsto \Omega^k_{\mathrm{con}(q)}(U_1\cup U_2)\stackrel{\Phi}{\to}\Omega^k_{\mathrm{con}(q)}(U_1)\oplus\Omega^k_{\mathrm{con}(q)}(U_2)\stackrel{\Psi}{\to}\Omega^k_{\mathrm{con}(q)}(U_1\cap U_2)\to 0
\end{equation}
		where  
		\[ \Phi(\omega)=(\omega|_{U_1},\omega|_{U_2}),\quad \Psi(\omega_1,\omega_2)=\omega_2|_{U_1\cap U_2}-\omega_1|_{U_1\cap U_2}.\]
		This sequence is exact: Given $(\omega_1,\omega_2)\in\mathrm{ker}(\Psi)$, we get a unique differential from $\omega$ restricting to $\omega_1$ and $\omega_2$ on $U_1$ and $U_2$, respectively. It is constructible because
		it is piecewise constructible. The behaviour of the derivatives can be tested locally. Hence $\omega\in\Omega^k_{\con(q)}(U_1\cup U_2)$ as claimed. 
		
		Now consider $\eta\in\Omega^k_{\con(q)}(U_1\cap U_2)$.
		Choose according to Proposition \ref{partition-of-unity} globally subanalytic $C^p$-functions $f_1,f_2:U_1\cup U_2\to \IR_{\geq 0}$ with $f_1+f_2=1$ and $\mathrm{supp}(f_l)\subset U_l$ for $l\in \{1,2\}$. Let $V_l\subset U_1\cup U_2$ be the complement of the support of $f_l$. We put
		\[ \eta_2=\begin{cases} f_1\eta &\text{on $U_1\cap U_2$}\\
			0 & \text{on $U_2-U_1\subset V_1$}
		\end{cases}
		\]
		on $U_2$
		and analogously $\eta_1$ on $U_1$. Then $(-\eta_1,\eta_2)\in\Omega^k_{\con(q)}(X)$ and
		\[ \Psi(-\eta_1,\eta_2)=f_1\eta-(-f_2\eta)=\eta.\]
		
%		Adapting the classical reasoning (see for example \cite[Chapter I, \S 5 and Chapter II, \S 9]{bott_tu}) \annette{Das steht da aber nicht} we obtain that the constructible de Rham theorem holds for $U_1\cup U_2$. 

Denote by $S^\bullet(X)$ the complex defining $H^\bullet_{\sing,\sub}(X,\R)$. By functoriality of the pairing in Definition~\ref{defn:de Rham theorem}, we get a morphism from the short exact sequence of complexes in \eqref{MV} to the (exact up to chain homotopy equivalence) sequence, see \cite[Proposition~8.1 (d) and Proposition~8.6]{dold}
\[ 0\to S^\bullet(U_1\cup U_2)\to S^\bullet(U_1)\oplus S^\bullet(U_2)\to S^\bullet(U_1\cap U_2)\to 0.\]
This induces a map between the long exact Mayer-Vietoris sequences for $H^\bullet_{\dR,\con(q)}(\cdot)$ and $H^\bullet_{\sing,\sub}(\cdot,\R)$. The de Rham theorem for $U_1\cup U_2$ follows by the 5-lemma from the de Rham theorem for
$U_1$, $U_2$ and $U_1\cap U_2$. 
	\end{proof}

	\begin{theorem}[Main Theorem]
	\label{T 6.4}
	Assume that $p<\infty$. 
	The constructible de Rham theorem holds for any globally subanalytic $C^p$-manifold. 
\end{theorem}

\begin{proof}
	Let $X$ be a globally subanalytic $C^p$-manifold.
	
	We first consider the case that $X$ is a globally subanalytic open subset of $\R^n$ with the induced manifold structure. We argue by induction on $n$. The case
	$n=0$ is trivial. 
	
	Now assume that the assertion holds for globally subananlytic open subsets of $\R^{n-1}$. 
	By considering the globally subanalytic $C^\omega$-isomorphism
	$$\IR^n\to (-1,1)^n, x=(x_1,\ldots,x_n)\mapsto \big(x_1/\sqrt{1+x_1^2},\ldots,x_n/\sqrt{1+x_n^2}\big),$$
	we can assume that $X$ is bounded.
	By Wilkie \cite{wilkie_open} we have that $X$ is a finite union of globally subanalytic $C^0$-cells $V_1,\dots,V_s$. We view them as ribbons and argue by induction on the number of ribbons. The statement is empty for $s=0$. The ribbon
	$V_s$ is by Proposition \ref{P 5.7} globally $C^p$-homotopy equivalent to an open globally subanalytic set in $\R^{n-1}$. By inductive hypothesis (with respect to $n$) and Remark \ref{R 6.2}, the constructible de Rham theorem holds for $V_s$. By inductive hypothesis (with respect to $s$), the construtcible de Rham theorem holds for
	$U=V_1\cup\dots \cup V_{s-1}$. Moreover, 
	\[ U\cap V_s=(V_1\cap V_s)\cup \dots \cup (V_{s-1}\cap V_s)\]
	is itself a union of $s-1$ ribbons by Lemma~\ref{L 5.6}. By the Mayer-Vietoris property, see Lemma~\ref{L 6.3} the constructible de Rham theorem holds for $X=U\cup V_s$. 
	This settles the case of open subsets of $\R^n$.
	
	For the general case let $X$ be a globally subanalytic $C^p$-manifold with finite atlas
	$(\phi_i:U_i\to W_i)_{1\leq i\leq r}$. Note that by above the constructible de Rham theorem holds for $W_i$ and the diffeomorphic $U_i$. We argue by induction on $r$. Let $U=U_1\cup\dots\cup U_{r-1}$. Note that $U\cap U_r \subset U_r$
	is itself diffeomorphic to a globally subanalytic open subset of $\R^n$, hence the constructible de Rham theorem holds for $U\cap U_r$ by the first case. Applying the inductive hypothesis to $U$, we deduce the constructible de Rham theorem for $X=U\cup U_r$ by
	Lemma~\ref{L 6.3}.
\end{proof}

		\begin{cor}
		Let $p<\infty$ and let $0\leq q\leq p-1$.
		We have that $\dim (H^k_{\mathrm{dR},\mathrm{con}(q)}(X))<\infty$ for all $k\in \IN_0$.
	\end{cor}
	
	\begin{proof}
		We have by Theorem \ref{T 6.4} that $\dim (H^k_{\mathrm{dR},\mathrm{con}(q)}(X))=b_k(X)$ where $b_k$ denotes the $k$-th Betti number of $X$. These are finite (since sets definable in o-minimal structure can be finitely triangulated). 
	\end{proof}
	
	Given a globally subanalytic $C^\omega$-manifold $X$
	we have the classical de Rham cohomology groups $H_\mathrm{dR}^\bullet(X)$ stemming from $C^\infty$-differential forms and the real analytic de Rham cohomology groups $H_{\mathrm{dR},\omega}^\bullet(X)$.
	Note that by Beretta \cite{B} the canonical imbedding gives an isommorphism $H_{\mathrm{dR},\omega}^\bullet(X)\stackrel{\sim}{\to} H_\mathrm{dR}^\bullet(X)$.

	\begin{cor}
		Assume that $X$ is a globally subanalytic $C^\omega$-manifold.
Let $0\leq q<\infty$.
		We have a natural isomorphism $H^\bullet_{\mathrm{dR},\mathrm{con}(q)}(X)\stackrel{\sim}{\to}H^\bullet_\mathrm{dR}(X)$.
	\end{cor}
	
	\begin{proof}
		Let $k\in \IN_0$.
		As mentioned above the canonical map
		$\alpha:H_k^{\sing,\sub}(X,\R)\to H_k^\sing(X,\R)$ is an isomorphism, which gives an isomorphism 
		$$\beta:H^k_\sing(X,\R)\to H^k_{\sing,\sub}(X), f\mapsto f\circ \alpha.$$
		By Theorem \ref{T 6.4} the map
		$$\gamma: H^k_{dR,\mathrm{con}(q)}(X) \to H^k_{\sing,sub}(X), [\omega]\mapsto ([\sigma]\mapsto \int_\sigma \omega)$$
is an isomorphism.
		By the classical de Rham theorem the map
		$$\delta:H^k_\mathrm{dR}(X)\to H^k_\sing(X,\R), [\omega]\mapsto ([\sigma]\mapsto \int_\sigma \omega)$$
is an isomorphism.
		(Note that every ordinary singular cohomology class has a $C^1$-re\-pre\-sen\-ta\-tive.)
		Then
		$$\varphi:=\delta^{-1}\circ\beta^{-1}\circ \gamma: H^k_{\mathrm{dR},\mathrm{con}(q)}(X)\to H^k_\mathrm{dR}(X)$$
		is an isomorphism.
	\end{proof}
	
	In the compact case, we can formulate also the case $p=\omega$. Here even the globally subanalytic de Rham theorem (defined analogously) holds. 
	
	\begin{rem}
	Let $X$ be a globally subanalytic $C^\omega$-manifold that is compact.
	The following holds:
	\begin{itemize}
		\item[(1)] The constructible de Rham theorem holds for $X$. In particular, we have a natural isomorphism $H_{\mathrm{dR},\mathrm{con}(\omega)}^\bullet(X)\stackrel{\sim}{\to} H_\mathrm{dR}^\bullet(X)$.
		\item[(2)] The globally subanalytic de Rham theorem holds for $X$. In particular, we have a natural isomorphism $H_{\mathrm{dR},\mathrm{sub}(\omega)}^\bullet(X)\stackrel{\sim}{\to} H_\mathrm{dR}^\bullet(X)$.
	\end{itemize}
	\end{rem}

	\begin{proof}
	This follows from the above mentioned result of \cite{B} and the fact that a real analytic function on a compact globally subanalytic $C^\omega$-manifold is globally subanalytic.
	\end{proof}
	
	\begin{rem}

		Note that the second statement of the previous remark does not hold in the non-compact case (see Example \ref{not computing}).
		It remains open whether the first statement holds in the non-compact case.

	\end{rem}

\section{Sheaf-theoretic approach}

In this section, we give another proof of the de Rham theorem in the constructible setting from the point of view of sheaves on the definable site.

We begin with some generalities on the cohomology of sheaves on the definable site of a definable manifold. 
We recall that $\mathcal{M}$ is an arbitrary o-minimal expansion of the ordered field of real numbers. We shall only later specialize to the case that $\mathcal{M} = \R_\an$.

\emph{Throughout this section, we let $(X,[\mathcal{A}])$ denote a definable $C^p$-manifold of dimension $n$.}

\subsection{Sheaves on the definable site}

We follow the approach of Edmundo, Jones and Peatfield in \cite{edmundo-jones-peatfield}.  
Recall that we denote by $\mathrm{Def}(X)$ the (finitary) Boolean algebra of definable subsets of $X$.

\begin{defn}[The definable site $X_\mathrm{def}$]
    The \emph{definable site $X_\mathrm{def}$ of $X$} is the Grothendieck topology on $X$ with the admissible open subsets being the definable open subsets $U$ of $X$, and admissible coverings of a definable open $U$ are defined to be the \emph{finite} coverings by definable open subsets. It is not hard to verify that $X_\mathrm{def}$ indeed satisfies the axioms of being a Grothendieck topology.
\end{defn}

\begin{notation}
    The category of sheaves of abelian groups on $X_\mathrm{def}$ will be denoted by $Sh(X_\mathrm{def})$. For a sheaf of abelian groups $\mathcal{F}$ on $X_\mathrm{def}$, we denote by $H^i(X_\mathrm{def},\mathcal{F})$ the $i^\mathrm{th}$ right-derived functor of the left-exact global sections functor \begin{align*}
       \Gamma(X, \cdot) :  Sh(X_\mathrm{def}) &\rightarrow \IZ\text{-mod} \\
       \mathcal{F} &\mapsto \mathcal{F}(X).
    \end{align*}
\end{notation}

\begin{rem}[The definable spectrum \`a la \cite{edmundo-jones-peatfield}.]
    We recall the construction in \cite{edmundo-jones-peatfield} of the definable spectrum $\widetilde{X}$ of a definable $C^p$-manifold $X$. 
    
    The underlying set of the definable spectrum is defined as follows: \[\widetilde{X} := \{\mathfrak{q} \subset \mathrm{Def}(X) : \mathfrak{q} \text{ is an ultrafilter on } \mathrm{Def}(X)\}.\]  
    For a definable subset $W \subset X$, we denote by $\widetilde{W} \subset \widetilde{X},$ the subset $\widetilde{W} := \{\mathfrak{q} \in \widetilde{X} : W \in \mathfrak{q}\}.$  
    One may verify that for a finite collection $\{W_i : 1\leq i \leq k\}$ of definable subsets of $X$, $\widetilde{(\bigcup_{i=1}^k W_i)} = \bigcup_{i=1}^k \widetilde{W_i}$  and that $\widetilde{(\cap_{i=1}^k W_i)} = \cap_{i=1}^k \widetilde{W_i}$. 

    The sets of the form $\widetilde{U}\subset \widetilde{X}$ for definable opens $U \subset X$, form a basis for a topology on $\widetilde{X}$. Under this topology $\widetilde{X}$ becomes a spectral topological space, with the sets of the form $\widetilde{U}$ for open definable subsets $U \subset X$ forming a basis of quasi-compact open subsets of $\widetilde{X},$ stable under finite intersections. 

    Every sheaf of abelian groups $\mathcal{F}$ on $X_\mathrm{def}$ extends uniquely to a sheaf of abelian groups $\widetilde{\mathcal{F}}$ on $\widetilde{X}$, such that for a definable open $U \subset X$, $\widetilde{\mathcal{F}}(\widetilde{U}) = \mathcal{F}(U).$ In fact, the association $\mathcal{F}\mapsto \widetilde{\mathcal{F}}$ is functorial in $\mathcal{F}$ and establishes an equivalence of categories between the category of sheaves of abelian groups on $X_\mathrm{def}$ and the category of sheaves of abelian groups on the topological space $\widetilde{X}.$ We therefore see that $H^i(X_\mathrm{def},\mathcal{F})$ is canonically isomorphic to the sheaf cohomology $H^i(\widetilde{X},\widetilde{\mathcal{F}})$ computed on the spectral topological space $\widetilde{X}.$
\end{rem}

\begin{lemma}\label{lem:def.flasque-acyclic}
    Let $\mathcal{F} \in Sh(X_\mathrm{def})$ be an abelian sheaf that is flasque on the definable site (that is for every inclusion of definable open subsets $U'\subset U$, the restriction map $\mathcal{F}(U)\rightarrow \mathcal{F}(U')$ is surjective). Then $\mathcal{F}$ is acyclic for the global sections functor, that is for every $k \geq 1$, $H^k(X_\mathrm{def},\mathcal{F}) = 0$. 
\end{lemma}
\begin{proof}
	
	The argument for topological spaces works with little change. We also refer the reader to \cite[Exp. V, \S 4.8 and Ex. 4.16]{SGA4}.

\end{proof}

\subsection{Comparing definable sheaf with singular cohomology}

For an abelian group $G$, we shall denote by $\underline{G}_{X_\mathrm{def}}$ the constant sheaf induced by $G$ on the definable site, $\underline{G}_X$ the constant sheaf on the topological space $X$, and finally $\widetilde{\underline{G}}$ the constant sheaf induced by $G$ on $\widetilde{X}.$ 

We have morphisms of sites\footnote{Note the direction of the arrows. We follow \cite[\href{https://stacks.math.columbia.edu/tag/00X1}{Tag 00X1}]{stacks-project} for our conventions on morphisms of sites.} $i : X_\mathrm{top} \rightarrow X_\mathrm{def}$, and $j : \widetilde{X}_\mathrm{top} \rightarrow X_\mathrm{def}.$ 
Note that by $X_\mathrm{top}$ (respectively by $\widetilde{X}_\mathrm{top}$) here we mean the site of open subsets of $X$ ($\widetilde{X}$ respectively) with covers being arbitrary open covers.
It is not hard to see that $i_\ast\underline{G}_X = \underline{G}_{X_\mathrm{def}}$ \footnote{essentially since connected components of a definable open subset are again definable and since a definable open subset is connected iff it is definably connected}, and $j^{-1}(\underline{G}_{X_\mathrm{def}}) = \widetilde{\underline{G}_{X_\mathrm{def}}} = \underline{\widetilde{G}}.$

In the following we shall denote by $H^k_\mathrm{sing}(X,G)$, the $k^\mathrm{th}$-singular cohomology group of $X$ with coefficients in $G$, and by $H^k_\mathrm{sing, def}(X,G)$ the $k^\mathrm{th}$ o-minimal singular cohomology group of the definable manifold $X$ with coefficients in $G$ (see \cite[\S 5, 6]{edmundo_woerheide_comparison} and \cite{huber_period_iso_tame}.) Recall that $H^k_\mathrm{sing,def}(X,G)$ is computed using the complex of $G$-valued functions on the free abelian group of continuous definable simplices in $X$.

\begin{prop}
    For each $k \geq 0$, we have natural isomorphisms \[H^k(X_\mathrm{def},\underline{G}_{X_\mathrm{def}})\cong H^k_\mathrm{sing}(X,G) \cong H^k(X,\underline{G}_X) \cong H^k_\mathrm{sing,def}(X,G).\]    
\end{prop}
\begin{proof} 
	The second isomorphism above is standard, while the comparison  $H^k_\mathrm{sing}(X,G) \cong H^k_\mathrm{sing,def}(X,G)$ follows for instance from \cite{edmundo_woerheide_comparison} or \cite{huber_period_iso_tame}. The existence of a natural comparison isomorphism $H^k(X_\mathrm{def},\underline{G}_{X_\mathrm{def}})\cong H^k_\mathrm{sing}(X,G)$ follows from \cite{edmundo-jones-peatfield} and \cite{edmundo_woerheide_comparison}. We also sketch a direct argument below.
    
    For an open subset $U \subset X$, and $k \geq 0$ let $S^k_G(U)$ denote the abelian group of $G$-valued functions on the set of continuous $q$-simplices in $U$, and let $\mathfrak{S}^k_G$ denote the sheafification of the flasque presheaf $U \mapsto S^k_G(U)$.
    By \cite[Prop. 5.27, \S 5.31]{warner-foundations} the complex  \[0\rightarrow \underline{G}_X \rightarrow \mathfrak{S}^0_G \rightarrow \mathfrak{S}^1_G \rightarrow \ldots\] is a \emph{flasque resolution} of the constant sheaf $\underline{G}_X$ on $X$. 
    
    %For an inclusion of open subsets $U'\subset U$ of $X$, we have natural \emph{surjective} restriction maps $S^q_G(U) \rightarrow S^q_G(U')$, making $S^q_G$ a flasque presheaf on $X$. We let $\mathfrak{S}^q_G$ denote the sheaf on $X$ associated to the presheaf $S^q_G$. 
    
    %It follows for instance from \cite[Prop. 5.27]{warner-foundations} that each $\mathfrak{S}^q_G$ is flasque on $X$ and from \cite[\S 5.31]{warner-foundations} that $0\rightarrow \underline{G}_X \rightarrow \mathfrak{S}^0_G \rightarrow \mathfrak{S}^1_G \rightarrow \ldots$ is a resolution of the constant sheaf $\underline{G}_X$ on $X$. 
    
    %It is clear that each $i_\ast(\mathfrak{S}^q_G)$ is flasque on the definable site, hence by \autoref{lem:def.flasque-acyclic}, acyclic. 
    The pushforward complex \[0 \rightarrow \underline{G}_{X_\mathrm{def}} \rightarrow i_\ast(\mathfrak{S}^0_G) \rightarrow i_\ast(\mathfrak{S}^1_G) \rightarrow \ldots\] is then a flasque (hence acyclic by \autoref{lem:def.flasque-acyclic}) resolution of $\underline{G}_{X_\mathrm{def}}$ on the definable site. Indeed, one may argue the exactness of the pushforward complex by hand, using \cite{wilkie_open} to note that every definable open $U \subset X$ is a finite union of open $C^0$-cells which are in turn contractible (see \autoref{C 5.8})  Hence, $H^k(X_\mathrm{def},\underline{G}_{X_\defin}) \cong H^k(\Gamma(X,\mathfrak{S}^\bullet_G)) \cong H^k_\mathrm{sing}(X,G).$
    
	%    and is thus a flasque resolution (hence acyclic resolution by \autoref{lem:def.flasque-acyclic}) in $Sh(X_\mathrm{def})$ of $\underline{G}_{X_\mathrm{def}}.$ 
    
    %Let $U \subset X$ be a \emph{definable} open subset,  and $\sigma \in \mathfrak{S}^q_G(U)$ that maps to zero in $\mathfrak{S}_G^{q+1}(U).$ It follows from \cite{wilkie_open}, that $U$ is a \emph{finite union} of open definable cells, $U = \bigcup_{i=1}^r U_i$ where the $U_i$ are in particular contractible in the usual topology. Since for each $q \geq 1$, $H^q_\mathrm{sing}(U_i,G) = 0,$ we may find $\tau_i \in \mathfrak{S}^{q-1}_G(U_i)$ such that $\tau_i$ maps to $\sigma\vert_{U_i}.$ 
    This 
    %proves the required exactness of the complex of sheaves above on the definable site, and hence 
    completes the proof of the proposition.
\end{proof}

\subsection{Definable partitions of unity}

%Throughout this subsection, $X$ shall denote a definable $C^p$-manifold of dimension $n$, $X_\mathrm{def}$ the corresponding definable site, and $\widetilde{X}$ the definable spectrum of $X.$

\begin{defn}
    Let $\mathcal{A}$ be a sheaf of rings on $X_\mathrm{def}.$ We say that \emph{$\mathcal{A}$ admits definable partitions of unity}, if given any open definable subset $U \subset X$ and any finite cover of $U$ by open definable subsets $U = \bigcup_{i=1}^r U_i$ one can find sections $f_i \in \mathcal{A}(U), 1\leq i \leq r$, such that
    \begin{enumerate}
        \item  $\sum_{i=1}^r f_i = 1$ and
        \item the support of each $f_i$ is contained in some closed \emph{definable} subset of $U$ contained in $U_i$.
    \end{enumerate}
\end{defn}

\begin{ex}[The sheaf $\mathcal{C}^q_{\mathrm{def}}$ of definable $C^q$-functions]
	For $0 \leq q \leq p$, and a definable open subset $U \subset X$, recall that $\mathcal{C}^q_{\mathrm{def}}(U)$ denotes the $\R$-algebra of definable $C^q$-maps $U \rightarrow \R.$ For an inclusion $U' \subset U$ of definable open subsets, and a section $f \in \Cdef^q(U)$, the restriction $f\vert_{U'} : U'\rightarrow \R$ is a section of $\Cdef^q(U').$ The restriction maps $\Cdef^q(U) \rightarrow \Cdef^q(U')$ make $\Cdef^q$ a sheaf of $\R$-algebras on the definable site $X_\mathrm{def}$ of $X.$  From \autoref{partition-of-unity}, we see that \emph{when $\mathcal{M}$ is polynomially bounded and admits $C^\infty$-cell decomposition and when $q < \infty$}, the sheaf of rings $\Cdef^q$ on $X_\mathrm{def}$ admits definable partitions of unity. In particular, this is the case for $\R_\an$.
\end{ex}

\begin{lemma}
    Let $\mathcal{A}$ be a sheaf of rings on $X_\mathrm{def}$ admitting definable partitions of unity. Let $\mathcal{F}$ be a sheaf of $\mathcal{A}$-modules on $X_\mathrm{def}.$ Then $\mathcal{F}$ is acyclic for the global sections functor, that is $H^k(X_\mathrm{def},\mathcal{F}) = 0$ for all $k \geq 1$.
\end{lemma}
\begin{proof} We adapt the proof in \cite[Prop. 4.36]{voisin-hodgeI}.
    We let $\widetilde{\mathcal{A}}$ (resp. $\widetilde{\mathcal{F}}$) denote the sheaf of rings (resp. sheaf of $\widetilde{\mathcal{A}}$-modules) induced by $\mathcal{A}$ (resp. $\mathcal{F}$) on the definable spectrum $\widetilde{X}.$ We pick an injective resolution $0 \rightarrow \widetilde{\mathcal{F}} \rightarrow \mathcal{I}^{(0)} \xrightarrow{d^{0}} \mathcal{I}^{(1)} \xrightarrow{d^{1}} \ldots $ in the category of sheaves of $\widetilde{\mathcal{A}}$-modules on $\widetilde{X}.$ We note that $\mathcal{I}^{(i)}$ are flasque \cite[\href{https://stacks.math.columbia.edu/tag/01EA}{Lemma 01EA}]{stacks-project} and hence acyclic for the global sections functor evaluated on $Sh(\widetilde{X})$. Therefore, for $k \geq 1$, 
    \[H^k(X_\mathrm{def},\mathcal{F}) \cong \frac{\ker(\Gamma(\widetilde{X},\mathcal{I}^{(k)})\xrightarrow{d^k}\Gamma(\widetilde{X},\mathcal{I}^{(k+1)}))}{d^{k-1}(\Gamma(\widetilde{X},\mathcal{I}^{(k-1)}))}.\] Given $\alpha \in \ker(d^{k})$, using the quasi-compactness of $\widetilde{X}$, we see that there is a \emph{finite} definable open cover $X = \bigcup_{i=1}^l U_i$, and sections $\beta_i \in \Gamma(\widetilde{U_i},\mathcal{I}^{(k-1)})$, such that $d^{k-1}(\beta_i) = \alpha\vert_{\widetilde{U_i}}$. Pick a definable partition of unity $\{f_i \in \Gamma(X,\mathcal{A}): 1\leq i \leq l\}$ subordinate to the cover $\{U_i: 1\leq i \leq l\}.$ 
    For each $i$, we set $V_i$ to be the complement in $X$ of the support of $f_i$, so that $X = U_i \cup V_i$ is a definable open cover of $X$, and $f_i \vert_{V_i} = 0.$
    Let $\gamma_i \in \Gamma(\widetilde{X},\mathcal{I}^{(k-1)})$ be the unique section such that $\gamma_i\vert_{\widetilde{U_i}} = f_i\cdot \beta_i$, and $\gamma_i\vert_{\widetilde{V_i}} = 0.$ Setting $\gamma := \sum_{i=1}^l \gamma_i \in \Gamma(\widetilde{X},\mathcal{I}^{(k-1)})$ one may verify that $d^{(k-1)}(\gamma) = \alpha.$
    This proves that for $k \geq 1$, $H^k(X_\mathrm{def},\mathcal{F}) = 0.$  
\end{proof}

\begin{rem}[The sheaves $\Eh^i_{\defin(q)}$ and $\Omega_{\defin(q)}^i$]
Suppose that $p\geq 1$, $0 \leq q \leq p-1$ and $i \geq 0$. For each open definable subset $U \subset X$, we have defined earlier the subspaces $\Omega^i_{\defin(q)}(U) \subset \Eh^i_{\defin(q)}(U)$ (see \autoref{def:Omega-def-con-forms} and \autoref{def:Omega-con-forms-q=0}). 
The association $U \mapsto \Eh^i_{\defin(q)}(U)$ (respectively $U \mapsto \Omega^i_{\defin(q)}(U)$) endowed with the usual restriction maps gives rise to a sheaf $\Eh^i_{\defin(q)}$ (respectively $\Omega^i_{\defin(q)}$)  of $\Cdef^{q}$-modules (respectively of $\Cdef^{q+1}$-modules, see Remark~\ref{R 2.5} (1)) on the definable site $X_\mathrm{def}.$
Furthermore, the exterior derivative operator (see \autoref{def:exterior-derivative-Omega-q=0} for the case $q=0$) defines a complex of sheaves of abelian groups on $X_\mathrm{def}$, \[ 0 \rightarrow \underline{\R}_{X_\mathrm{def}} \rightarrow \Cdef^{q+1} \xrightarrow{D} \Omega^1_{\defin(q)} \xrightarrow{D} \Omega^{2}_{\defin(q)} \xrightarrow{D} \ldots \xrightarrow{D} \Omega^n_{\defin(q)} \rightarrow 0. \]

\end{rem}

\begin{cor}\label{cor:acyclicity-of-Cqdef-modules}
      Let $\mathcal{M}$ be a polynomially bounded o-minimal structure admitting $C^\infty$-cell decomposition. Let $q < \infty$.
      Then for any sheaf $\mathcal{F}$ of $\Cdef^q$-modules and $k \geq 1$, $H^k(X_\mathrm{def},\mathcal{F}) = 0.$ In particular, for $i,k, p \geq 1$, and $q \leq p-1$ with $q < \infty$, we have $H^k(X_\mathrm{def},\Eh^{i}_{\defin(q)}) = 0,$ and $H^k(X_\mathrm{def},\Omega^i_{\defin(q)}) = 0.$ 
\end{cor}

\subsection{Comparing constructible de Rham with singular cohomology}

\emph{Henceforth, we shall specialize to the setting where the o-minimal structure $\mathcal{M}$ under consideration is $\R_\an$.}
In particular, henceforth $(X,[\mathcal{A}])$ shall denote a globally subanalytic $C^p$-manifold of dimension $n$ for some $p \in \IN \cup \{\omega\}.$ Note that in particular $p \geq 1$.

\begin{notation}
	In the case that $\mathcal{M} = \R_\an$, we denote by $X_\mathrm{sub}$ the definable site $X_\defin$, by
	$\Csub^q$ the sheaf $\Cdef^q$, by $\Eh^{i}_{\sub(q)}$ the sheaf $\Eh^{i}_{\defin(q)}$, and by $\Omega^i_{\sub(q)}$ the sheaf $\Omega^i_{\defin(q)}$ defined above. 
\end{notation}

%We note that $\R_\an$ is a polynomially bounded o-minimal structure that admits $C^\infty$-cell decomposition, and hence the results of the previous sections, on definable partitions of unity also apply to this setting.

%In this globally subanalytic setting, Cluckers--Miller \cite{cluckers_miller}, have defined and studied the class of constructible functions. This is (roughly speaking) the smallest class of functions, containing the class of globally subanalytic functions that is also stable under parametrized integration. We refer the reader to loc. cit. for background on the theory of constructible functions.

\begin{rem}[The sheaves $\Ccon^q, \Eh^i_{\con(q)}$ and $\Omega^i_{\con(q)}$]
	\begin{itemize}
		\item Let $0 \leq q \leq p$. The association $U \mapsto \Ccon^q(U)$, the $\R$-algebra of constructible real valued $C^q$-functions on $U$,  along with usual restriction maps makes $\Ccon^q$ a sheaf of $\Csub^q$-algebras on $X_\mathrm{sub}$.
		\item Let $0 \leq q \leq (p-1)$. We have defined in \autoref{def:const-Cq-forms}, \autoref{def:Omega-def-con-forms} and \autoref{def:Omega-con-forms-q=0}, the subspaces $\Omega^i_{\con(q)}(U) \subset \Eh^i_{\con(q)}(U).$ The association $U \mapsto \Eh^i_{\con(q)}(U)$ (respectively $U \mapsto \Omega^i_{\con(q)}(U)$) along with the usual restriction maps defines a sheaf $\Eh^i_{\con(q)}$ (respectively $\Omega^i_{\con(q)}$) of $\Ccon^q$-modules (respectively of $\Ccon^{q+1}$-modules) on the subanalytic site $X_\sub$. 
		\item The exterior derivative defined in \autoref{def:exterior-derivative-Omega-q=0} gives rise to a complex $\Omega^\bullet_{X,\con(q)}$ of sheaves of abelian groups on $X_\mathrm{sub}$: 
		\[ 0 \rightarrow \underline{\R}_{X_\mathrm{sub}} \rightarrow \Oconext{0}{q} \xrightarrow{D} \Oconext{1}{q} \xrightarrow{D} \ldots \xrightarrow{D} \Oconext{n}{q} \rightarrow 0. \] 
	\end{itemize}

\end{rem} 

%We shall refer to this complex of sheaves $\Oconext{\bullet}{q}$ as the $C^q$-constructible de Rham complex of $X$. 

By \autoref{functoriality_easy} and \autoref{thm:functorial}, we see that given a globally subanalytic $C^p$-map $f: X \rightarrow Y$ between globally subanalytic $C^p$-manifolds $X$ and $Y$, the pullback of differential forms induces for each $0 \leq q \leq (p-1)$, a morphism of complexes of sheaves of abelian groups on $Y_\mathrm{sub}$: \[\Omega^\bullet_{Y,\con(q)} \rightarrow f_\ast(\Omega^\bullet_{X,\con(q)}).\] 

The $i^\mathrm{th}$-cohomology group of the complex $\Gamma(X_\mathrm{sub},\Oconext{\bullet}{q})$ of global sections is the $i^\mathrm{th}$ $C^q$-constructible de Rham cohomology group of $X$, denoted by $H^i_{\mathrm{dR,con}(q)}(X)$ earlier in \autoref{def:de-rham-complex-definable-const}.

It follows from the Poincar\'e lemma for open $C^0$-cells, that is \autoref{Poincare}, along with Wilkie's  observation \cite{wilkie_open} that every globally subanalytic open is a finite union of open $C^0$-cells,
that for every $0 \leq q \leq p-1$ the above complex $\Omega^\bullet_{X,\con(q)}$ is a \emph{resolution} of the constant sheaf $\underline{\R}_{X_\mathrm{sub}}.$
Furthermore, by Remark~\ref{R 2.5}  each $\Oconext{i}{q}$  being a sheaf of $\Ccon^{q+1}$-modules on $X_\mathrm{sub}$, is in particular a sheaf of $\Csub^{q+1}$-modules. 
Therefore, whenever $q < \infty$, by \autoref{cor:acyclicity-of-Cqdef-modules} the above complex is an \emph{acyclic resolution} of the constant sheaf $\underline{\R}_{X_\mathrm{sub}}.$  Thus we have proved the following:

\begin{thm}\label{thm:main-sheaf-version}
    Let $X$ be a globally subanalytic $C^p$-manifold (for some $p \in \IN\cup\{\omega\}$). Then for all $0 \leq q \leq p-1$, we have functorial isomorphisms:
    \[ \mathbb{H}^i(X_\sub,\Omega^\bullet_{X,\con(q)}) \cong H^i(X_\mathrm{sub},\underline{\R}_{X_\mathrm{sub}}) \cong H^i_\mathrm{sing}(X,\R), \] wherein $\mathbb{H}^i(X_\sub,\Omega^\bullet_{X,\con(q)})$ denotes the $i^\mathrm{th}$ hypercohomology group of the complex $\Omega^\bullet_{X,\con(q)}$ on $X_\mathrm{sub}$. Furthermore, if $q < \infty$, we have functorial isomorphisms:
    \[ H^i_{\mathrm{dR,con}(q)}(X) \cong H^i(X_\mathrm{sub},\underline{\R}_{X_\mathrm{sub}}) \cong H^i_\mathrm{sing}(X,\R). \] 
\end{thm}

\begin{finrem}
	\label{FR 6.13}
	\begin{itemize}
		\item[(1)]
	One could generalize the constructible de Rham theorem to the bigger o-minimal structure $\IR_\an^K$ where $K$ denotes the field of real algebraic numbers (see \cite{kaiser_measures}).
	\item[(2)] 
	In the semialgebraic case over $\IR$ or over $\IQ$ one could obtain the constructible de Rham theorem by introducing some proper subclasses of the spaces of constructible differential forms (see \cite{kaiser_integrated_nash,kaiser_integrated_algebraic}).
	\item[(3)] 
	One could even approach the non-archimedean case in the globally subanalytic setting by the results of \cite{kaiser_integration_non-arch,kaiser_fundamenta}. 
\end{itemize}
 But details have to be checked.
	\end{finrem}

\bibliographystyle{alpha}
\bibliography{deRham}
\end{document}